\documentclass[11pt,reqno]{amsart}

\date{August 1, 2019}
\usepackage[parfill]{parskip}    % Activate to begin paragraphs with an empty line rather than an indent
\usepackage{graphicx}
\usepackage{amssymb}
\usepackage{amsmath}
\usepackage{amsthm}
\usepackage{mathabx}
\usepackage{commath}
\usepackage{epstopdf}

\usepackage{amsmath,amsfonts,amsthm,amssymb,amsxtra,dsfont,mathrsfs}
\usepackage{hyperref} 
\usepackage{colordvi}
\usepackage[usenames,dvipsnames]{color}

\usepackage[usenames]{xcolor}
\usepackage{url}

\usepackage[T1]{fontenc}
\usepackage[utf8]{inputenc}
\usepackage[ngerman, german, english]{babel} 
\usepackage{bbm}

 \numberwithin{equation}{section}

\newtheorem{thm}{Theorem}[section]
\newtheorem{cor}[thm]{Corollary}
\newtheorem{lem}[thm]{Lemma}
\newtheorem{prop}[thm]{Proposition}

\theoremstyle{remark}

\newtheorem{rem}[thm]{\bf Remark}

\theoremstyle{definition}

\newtheorem{assumption}[thm]{Assumption}
\newtheorem{defin}[thm]{Definition}
\newtheorem{example}[thm]{Example}

\newcommand{\R}{\mathbb{R}} % reelle
\newcommand{\eps}{\varepsilon}

\title[Energy asymptotics in the three-dimensional Brezis--Nirenberg problem]{Energy asymptotics in the three-dimensional Brezis--Nirenberg problem}

\author{Rupert L. Frank}

\address[Rupert L. Frank]{Mathematisches Institut, Ludwig-Maximilians-Universit\"at M\"unchen, Theresienstr. 39, 80333 M\"unchen, Germany, and Mathematics 253-37, Caltech, Pasa\-de\-na, CA 91125, USA}

\email{r.frank@lmu.de, rlfrank@caltech.edu}

\author{Tobias König}

\address[Tobias König]{Mathematisches Institut, Ludwig-Maximilians-Universit\"at M\"unchen, Theresienstr. 39, 80333 M\"unchen, Germany}

\email{tkoenig@math.lmu.de}

\author {Hynek Kova\v{r}\'{\i}k}

\address [Hynek Kova\v{r}\'{\i}k]{DICATAM, Sezione di Matematica, Universit\`a degli studi di Brescia, Italy}

\email {hynek.kovarik@unibs.it}

\thanks{\copyright\, 2019 by the authors. This paper may be reproduced, in its entirety, for non-commercial purposes.\\
Partial support through US National Science Foundation grant DMS-1363432 (R.L.F.) and Studienstiftung des deutschen Volkes (T.K.) is acknowledged. H.~K.  has been partially supported by Gruppo Nazionale per Analisi Matematica, la Probabilit\`a e le loro Applicazioni (GNAMPA) of the Istituto Nazionale di Alta Matematica (INdAM)}

\begin{document}

\begin{abstract}
For a bounded open set $\Omega\subset\R^3$ we consider the minimization problem
$$
S(a+\epsilon V) = \inf_{0\not\equiv u\in H^1_0(\Omega)} \frac{\int_\Omega (|\nabla u|^2+ (a+\epsilon V) |u|^2)\,dx}{(\int_\Omega u^6\,dx)^{1/3}}
$$
involving the critical Sobolev exponent. The function $a$ is assumed to be critical in the sense of Hebey and Vaugon. Under certain assumptions on $a$ and $V$ we compute the asymptotics of $S(a+\epsilon V)-S$ as $\epsilon\to 0+$, where $S$ is the Sobolev constant. (Almost) minimizers concentrate at a point in the zero set of the Robin function corresponding to $a$ and we determine the location of the concentration point within that set. We also show that our assumptions are almost necessary to have $S(a+\epsilon V)<S$ for all sufficiently small $\epsilon>0$. 
\end{abstract}

\maketitle

\section{\bf Introduction and main results}

\subsection{Setting of the problem}

In their celebrated paper \cite{BrNi} Br\'ezis and Nirenberg considered the problem of minimizing the quotient
$$
\mathcal S_a[u] := \frac{\int_\Omega (|\nabla u|^2+ a|u|^2)\,dx}{(\int_\Omega u^6\,dx)^{1/3}}
$$
over all $0\not\equiv u\in H^1_0(\Omega)$, where $\Omega\subset\R^3$ is a bounded open set and $a$ is a continuous function on $\Omega$. We denote the corresponding infimum by
$$
S(a) := \inf_{0\not\equiv u\in H^1_0(\Omega)} \mathcal S_a[u] \,.
$$
This number is to be compared with
$$
S := 3 \left( \frac{\pi}{2} \right)^{4/3} \,,
$$
the sharp constant \cite{Rod,Ro,Au,Ta} in the Sobolev inequality
\begin{equation}
\label{eq:sobolev}
\int_{\R^3} |\nabla u|^2\,dx \geq S \left( \int_{\R^3} u^6\,dx\right)^{1/3} \,,
\qquad u\in\dot H^1(\R^3) \,.
\end{equation}
One of the findings in \cite{BrNi} is that if $a$ is small (for instance, in $L^\infty(\Omega)$), then $S(a)=S$. This is in stark contrast to the case of dimensions $N\geq 4$ where the corresponding analogue of $S(a)$ (with the exponent $6$ replaced by $2N/(N-2)$) is always strictly below the corresponding Sobolev constant, whenever $a$ is negative somewhere.

This phenomenon leads naturally to the following notion due to Hebey and Vaugon \cite{HeVa}.

\begin{defin}
Let $a$ be a continuous function on $\Omega$. We say that $a$ is \emph{critical} in $\Omega$ if $S(a)=S$ and if for any continuous function $\tilde a$ on $\Omega$ with $\tilde a\leq a$ and $\tilde a\not\equiv a$ one has $S(\tilde a)<S(a)$. 
\end{defin}

Our goal in this paper is to compute the asymptotics of $S(a+\epsilon V)-S$ as $\epsilon\to 0$ for critical $a$ and to understand the behavior of corresponding minimizers. Here $V$ is a bounded function on $\Omega$, without any restrictions on its sign.

A key role in our analysis is played by the regular part of the Green's function and its zero set. To introduce these, we follow the sign and normalization convention of \cite{rey2}. If the operator $-\Delta+a$ in $\Omega$ with Dirichlet boundary conditions is coercive (which, in particular, is the case if $a$ is critical), then it has a Green's function $G_a$ satisfying
\begin{equation} \label{Ga-pde}
\left\{
\begin{array}{l@{\quad}l}
-\Delta_x\, G_a(x,y) + a(x)\, G_a(x,y) =  4\pi \, \delta_y & \quad \text{in} \ \ \Omega\,, \\
& \\
G_a(x,y) = 0  & \quad \text{on} \ \ \partial\Omega \,.
\end{array}
\right.  
\end{equation}
The regular part of $G_a$ is defined by 
\begin{equation} \label{ha-def}
H_a(x,y) := \frac{1}{|x-y|} - G_a(x,y)\, .
\end{equation}
It is well-known that for each $x\in\Omega$ the function $H_a(x,\cdot)$, which is originally defined in $\Omega\setminus\{x\}$, extends to a continuous function in $\Omega$ and we abbreviate
$$
\phi_a(x) := H_a(x,x) \,.
$$
It is well-known that the function $\phi_a$ is relevant for problems involving the critical Sobolev exponent, see, e.g., \cite{Sc} and \cite{Ba}. For the problem at hand, it was shown in \cite[Thm. 7]{Br} that if $\phi_a(x)<0$ for some $x\in\Omega$, then $S(a)<S$.  (In \cite{Br} this is attributed to Schoen \cite{Sc} and a work in preparation by McLeod.) Conversely, it was conjectured in \cite{Br} and proved by Druet in \cite{dr} that if $S(a)<S$, then $\phi_a(x)<0$ for some $x\in\Omega$. An alternative proof, assuming only continuity of $a$, is given in \cite{esp}. Thus, the (non-local) condition $\min_\Omega \phi_a<0$ is necessary and sufficient for $S(a)<S$, and replaces the (local) condition $\min_\Omega a <0$ in dimensions $N\geq 4$.

The above results imply that, if $a$ is critical, then $\min_\Omega \phi_a = 0$. In particular, the set
$$
\mathcal N_a := \{ x\in\Omega:\ \phi_a(x)=0 \}
$$ 
is non-empty.

%%%%%%%%%%%%%

\subsection{Main results}

Let us proceed to a precise statement of our main results. Throughout this paper we work under the following assumption.

\begin{assumption}\label{ass}
The set $\Omega\subset\R^3$ is open, bounded and has a $C^2$ boundary. The function $a$ satisfies $a\in C(\overline\Omega)\cap C^1(\Omega)$ and is critical in $\Omega$. Moreover,
\begin{equation}
\label{eq:anegative}
a(x)< 0
\qquad\text{for all}\ x\in\mathcal N_a \,.
\end{equation}
Finally, $V\in L^\infty(\Omega)$.
\end{assumption}

We will see in Corollary \ref{cor-2} that criticality of $a$ alone implies $a(x)\leq 0$ for all $x\in\mathcal N_a$. Therefore assumption \eqref{eq:anegative} is not severe.

We set
\begin{align} \label{eq-Q}
Q_V(x) & := \int_\Omega V(y) \, G_a(x,y)^2 \,dy , \qquad x\in\Omega \,,
\end{align}
and
$$
\mathcal N_a(V) :=\{ x\in \mathcal N_a:\ Q_V(x)< 0 \} \,.
$$

The following is our main result.

\begin{thm}
\label{thm-main}
Assume that $\mathcal N_a(V)\neq \emptyset$. Then $S(a+\epsilon V)<S$ for all $\epsilon>0$ and 
\begin{equation} 
\label{S-lim}
\lim_{\epsilon\to 0+} \frac{ S(a+\epsilon V)-S}{\epsilon^2} = \,  - \left( \frac 3S \right)^\frac12 \frac{1}{8\pi^2} \sup_{x \in \mathcal N_a(V)} \frac{Q_V(x)^2}{|a(x)|} \, .
\end{equation}
\end{thm}

We supplement this theorem with a result for the opposite case where $\mathcal N_a(V)=\emptyset$.

\begin{thm}
\label{thm-main2}
Assume that $\mathcal N_a(V)=\emptyset$. Then $S(a+\epsilon V) = S + o(\epsilon^2)$ as $\epsilon \to 0+$. If, in addition, $Q_V(x) > 0$ for all $x\in \mathcal N_a$, then $S(a+\epsilon V) = S$ for all sufficiently small $\epsilon >0$.
\end{thm}

It follows from the above two theorems that the condition $\mathcal N_a(V)\neq \emptyset$ is 'almost' necessary for the inequality $S(a+\epsilon V) <S$ for all small $\epsilon>0$. Only the case where $\min_{\mathcal N_a} Q_V=0$ is left open.

\begin{example}
When $\Omega=B$ is the unit ball in $\R^3$, then it is well-known that the constant function $a=-\pi^2/4$ is critical and that in this case $\mathcal N_a=\{0\}$ and $G_a(0,y)=|y|^{-1}\cos(\pi|y|/2)$; see, e.g., \cite{Br}. Thus, with
$$
q_V := Q_V(0) = \int_B V(y)\, \frac{\cos^2(\pi|y|/2)}{|y|^2}\,dy
$$
we have
$$
\lim_{\epsilon\to 0+} \frac{S(a+\epsilon V)-S}{\epsilon^2} = - \left( \frac 3S \right)^\frac12 \frac{1}{2\pi^4}  \, q_V^2
\qquad\text{if}\ q_V\leq 0
$$
and $S(a+\epsilon V)=S$ for all sufficiently small $\epsilon>0$ if $q_V>0$.
\end{example}

\begin{rem}
It is instructive to compare our results here with the results for the analogous problem
$$
S(\epsilon V) := \inf_{0\not\equiv u\in H^1_0(\Omega)} \frac{\int_\Omega (|\nabla u|^2+\epsilon Vu^2)\,dx}{\left( \int_\Omega |u|^{2N/(N-2)}\,dx \right)^{(N-2)/N}}
$$
in dimension $N\geq 4$. Let $S_N$ be the sharp constant in the Sobolev inequality in $\R^N$. From \cite{BrNi} we know that $S(\epsilon V)<S_N$ if and only if $V(x) <0$ for some $x\in \Omega$, and therefore we focus on the case where $\mathcal N(V):=\{ x\in\Omega:\ V(x)<0\}\neq\emptyset$. Then 
\begin{align} 
S(\epsilon V)  & =  S_N - C_N \sup_{x\in\mathcal N(V)} \frac{|V(x)|^{\frac{N-2}{N-4}}}{\phi_0(x)^\frac{2}{N-4}} 
\ \epsilon^{\frac{N-2}{N-4}} + o(\epsilon^{\frac{N-2}{N-4}}) &\!\! \text{\rm if}  \ \ N\geq 5 \,, \label{ngeq-5}\\
S(\epsilon V)  & = S_N - \exp\Big( - \frac 4\epsilon \left(1 +o(1)\right) \inf_{x\in\mathcal N(V)} \frac{\phi_0(x)}{|V(x)|}  \Big)  &  \text{  \ \ if}  \  N=4 \,,  \label{n4}
\end{align}
with explicit constants $C_N$ depending only on $N$. Note that, as a reflection of the Br\'ezis--Nirenberg phenomenon, $V$ enters pointwisely into the asymptotic coefficient in \eqref{ngeq-5} and \eqref{n4}, while it enters non-locally through $Q_V$ into the asymptotic coefficient in Theorem \ref{thm-main}.
\end{rem}

Asymptotics \eqref{ngeq-5} and \eqref{n4} in the case where $V$ is a negative constant are essentially contained in \cite{Tak}; see also \cite{We} for related results. The case of general $V\in C(\overline\Omega)$ can be treated by similar methods. We emphasize that the proof of Theorem \ref{thm-main} is considerably more complicated than that of \eqref{ngeq-5} and \eqref{n4}, since the expansion in Theorem \ref{thm-main} should rather be thought of as a higher order expansion of $S(a+\epsilon V)-S$ where the coefficient of the term of order $\epsilon$ vanishes due to criticality. In the higher dimensional context, no such cancellation occurs. 

%%%%%%%%%%%%%%%%%%%

\subsection{Behavior of almost minimizers}

We prove Theorems \ref{thm-main} and \ref{thm-main2} by proving upper and lower bounds on $S(a+\epsilon V)$. For the upper bound it suffices to evaluate $\mathcal S_{a+\epsilon V}[u_\epsilon]$ for an appropriately chosen family of functions $u_\epsilon$. For the lower bound we need to evaluate the same quantity where now $u_\epsilon$ is an optimizer for $S(a+\epsilon V)$. To do so, we will show that $u_\epsilon$ is essentially of the same form as the family chosen to prove the upper bound. In fact, we will not use the minimality of the $u_\epsilon$ and show that, more generally, all `almost minimizers' have essentially the same form as the functions chosen for the upper bound.

Given earlier works and, in particular, those by Druet \cite{dr} and Esposito \cite{esp} it is not surprising that almost minimizers concentrate at a point in the set $\mathcal N_a$. One of our new contributions is to show that this concentration happens at a point in the subset $\mathcal N_a(V)$ and, more precisely, at a point in $\mathcal N_a(V)$ where the supremum in \eqref{S-lim} is attained.

In order to state our theorem about almost minimizers, for $x \in \Omega$ and $\lambda > 0$, let
$$
U_{x, \lambda}(y) := \frac{\lambda^{1/2}}{(1 + \lambda^2|y-x|^2)^{1/2}} \,.
$$
The functions $U_{x,\lambda}$ and their multiples are precisely the optimizers of the Sobolev inequality \eqref{eq:sobolev}; see the references mentioned above and \cite[Cor. I.1]{Lions}. We introduce $PU_{x, \lambda} \in H^1_0(\Omega)$ as the unique function satisfying 
\begin{equation} \label{eq-pu}
\Delta PU_{x,\lambda} = \Delta U_{x,\lambda}\ \ \  \text{ in } \Omega, \qquad PU_{x,\lambda} = 0 \ \ \ \text{ on } \partial \Omega \,.
\end{equation}
Moreover, let
$$
T_{x, \lambda} := \text{ span}\, \big\{ PU_{x, \lambda}, \partial_\lambda PU_{x, \lambda}, \partial_{x_i} PU_{x, \lambda}\,  (i=1,2,3) \big\}
$$
and let $T_{x, \lambda}^\perp$ be the orthogonal complement of $T_{x,\lambda}$ in $H^1_0(\Omega)$ with respect to the inner product $\int_\Omega \nabla u \cdot \nabla v\,dy$. Finally, by $\Pi_{x,\lambda}$ and $\Pi_{x,\lambda}^\bot$ we denote the orthogonal projections in $H^1_0(\Omega)$ onto $T_{x,\lambda}$ and $T_{x,\lambda}^\bot$, respectively.

\begin{thm} \label{thm-minimizers}
Assume that $\mathcal N_a(V)\neq \emptyset$. Let $(u_\epsilon)\subset H^1_0(\Omega)$ be a family of functions such that
\begin{equation} \label{appr-min}
\lim_{\epsilon\to 0} \frac{\mathcal S_{a+\epsilon V}[u_\epsilon] - S(a+\epsilon V)}{S-S(a+\epsilon V)}  = 0 \qquad \text{and} \qquad \int_\Omega u_\epsilon^6 \,dx = \left( \frac S3\right)^{\frac 32} \,.
\end{equation}
Then there are $(x_\epsilon)\subset\Omega$, $(\lambda_\epsilon)\subset(0,\infty)$ and $(\alpha_\epsilon)\subset\R$ such that
\begin{equation} \label{u-eps-final}
u_\epsilon =  \alpha_\epsilon \left( PU_{x_\epsilon, \lambda_\epsilon} - \lambda_\epsilon^{-1/2}\, \Pi_{x_\epsilon,\lambda_\epsilon}^\bot (H_a(x_\epsilon, \cdot)- H_0(x_\epsilon, \cdot)) + r_\epsilon \right)
\end{equation}
and, along a subsequence,
\begin{align*}
x_\epsilon & \to x_0
\qquad\text{for some}\ x_0\in\mathcal N_a(V) \ \text{with}\ \ \frac{Q_V(x_0)^2}{|a(x_0)|} = \sup_{y \in \mathcal N_a(V)} \frac{Q_V(y)^2}{|a(y)|} \,, \\
\phi_a(x_\epsilon) & = o(\epsilon) \,,\\
\lim_{\epsilon \to 0}\,  \epsilon \,  \lambda_\epsilon & = 4\pi^2\, \frac{|a(x_0)|}{|Q_V(x_0)|} \,, \\
\alpha_\epsilon & = s + \mathcal O(\epsilon)
\qquad\text{for some}\ s\in\{\pm 1\} \,.
\end{align*}
Finally, $r_\epsilon\in T_{x_\epsilon,\lambda_\epsilon}^\bot$ and $\|\nabla r_\epsilon\|=o(\epsilon)$.
\end{thm}

The $L^6$ normalization in \eqref{appr-min} is chosen in view of
$$
\int_{\R^3} U_{x,\lambda}^6\,dy = \left( \frac S3\right)^{\frac 32} \,.
$$

There is a huge literature on blow-up results for solutions of equations involving the critical Sobolev exponent. Early contributions related to the problem we are considering are, for instance, \cite{AtPe,Bu,BrPe,Ha,rey1}; see also the book \cite{DrHeRo} for more recent developments and further references. Here we follow a somewhat different philosophy and focus not on the equation satisfied by the minimizers, but solely on their minimality property. Therefore our proofs also apply to almost minimizers in the sense of \eqref{appr-min} and we obtain blow-up results for those as well. On the other hand, with our methods we cannot say anything about non-minimizing solutions of the corresponding equation and our blow-up bounds are only obtained in $H^1$ instead of $L^\infty$ norm. Other related works which study Sobolev critical problems from a variational point of view are, for instance, \cite{FlGaMu,AmGa,Fl}.

As already mentioned before, the works of Druet \cite{dr} and Esposito \cite{esp}, and similarly \cite{FlGaMu,AmGa} in related problems, show that concentration happens at a point in $\mathcal N_a$. In terms of $S(a+\epsilon V)$, this corresponds essentially to the fact that $S(a+\epsilon V)= S + o(\epsilon)$. In order to go further than that and to compute the coefficient of $\epsilon^2$, we need to prove that concentration happens in the subset $\mathcal N_a(V)$ at a point where the supremum in \eqref{S-lim} is attained.

The strategy of the proof of the lower bound is to expand the quotient $\mathcal S_{a+\epsilon V}[u_\epsilon]$ for an almost minimizer $u_\epsilon$ as precisely as allowed by the available information on $u_\epsilon$, then to use a coercivity bound to deduce that certain terms are small and thereby improving our knowledge about $u_\epsilon$. We repeat this procedure three times (namely, in Sections \ref{sec:apriori}, \ref{sec-reloaded} and \ref{sec-new-decomp}). Therefore, a key tool in our analysis is the coercivity of the quadratic form
$$
\int_\Omega (|\nabla v|^2 + av^2 -15\, U_{x,\lambda}^4 v^2)\,dx \,,\qquad v\in T_{x,\lambda}^\bot \,,
$$
provided that $\lambda\,\mathrm{dist}(x,\partial\Omega)$ is sufficiently large; see Lemma \ref{coercivity}. This coercivity was proved by Esposito \cite{esp} and comes ultimately from the non-degeneracy of the Sobolev minimizer $U_{x,\lambda}$. Esposito used this bound to obtain an a priori bound on the term $\alpha_\epsilon^{-1} u_\epsilon - PU_{x_\epsilon,\lambda_\epsilon}$ in Theorem \ref{thm-minimizers}. We will use it for the same purpose in Proposition \ref{wbound}, but then we will use it two more times in Propositions \ref{thm-q} and in Lemma \ref{lem-g} in order to get bounds on $\alpha_\epsilon^{-1} u_\epsilon - PU_{x_\epsilon,\lambda_\epsilon}+\lambda^{-1/2} (H_a(x_\epsilon,\cdot)-H_0(x_\epsilon,\cdot))$ and $\alpha_\epsilon^{-1} u_\epsilon - PU_{x_\epsilon,\lambda_\epsilon}+\lambda^{-1/2}\, \Pi_{x,\lambda}^\bot (H_a(x_\epsilon,\cdot)-H_0(x_\epsilon,\cdot))$, respectively. After the last step we are able to compute the energy to within $o(\epsilon^2)$. We emphasize that in principle there is nothing preventing us from continuing this procedure and computing the energy to even higher precision.

Let us briefly comment on a surprising technical subtlety in our proof. While Theorem \ref{thm-minimizers} says that almost minimizers are essentially given by
$$
PU_{x, \lambda} - \lambda^{-1/2}\, \Pi_{x,\lambda}^\bot (H_a(x, \cdot)- H_0(x, \cdot))
$$
with $x\in\mathcal N_a(V)$ a maximum point for the right side in \eqref{S-lim} and $\lambda$ proportional to $\epsilon^{-1}$, to prove the upper bound we use the simpler functions
$$
PU_{x, \lambda} - \lambda^{-1/2} (H_a(x, \cdot)- H_0(x, \cdot))
$$
(with the same choices of $x$ and $\lambda$). The difference between the two functions, namely
$$
-\lambda^{-1/2}\, \Pi_{x,\lambda}(H_a(x,\cdot)-H_0(x,\cdot)) \,,
$$
can be shown to be of order $\epsilon$ (when $\lambda$ is proportional to $\epsilon^{-1}$), but not smaller; see Remark \ref{coeffsize}. Therefore it is not at all obvious that the two families of functions lead to the same (within $o(\epsilon^2)$) value of $\mathcal S_{a+\epsilon V}[\cdot]$. The fact that they do is contained in Lemma \ref{expfinalq}, where the contributions of $-\lambda^{-1/2}\, \Pi_{x,\lambda}(H_a(x,\cdot)-H_0(x,\cdot))$ to the numerator and to the denominator are shown to cancel each other to within $o(\epsilon^2)$.

At first sight, the problem considered in this paper resembles the problem of minimizing the quotient $\int_{\R^N} (|\nabla u|^p + \epsilon V |u|^p)\,dx / \int_{\R^N} |u|^p\,dx$ for $p\leq N$, which is a classical problem for $p=2$ \cite{Si} motivated by quantum mechanics and which was studied in \cite{EkFrKo} for general $p$. The underlying mechanism, however, is rather different. In these works almost minimizers spread out, whereas here they concentrate. The concentration regime is much more sensitive to the local details of the perturbation and necessitates, in particular, the use of orthogonality conditions  in $T_{x,\lambda}^\bot$ and the resulting coercivity.

%%%%%%%%%%%%%%%%%%%

%%%%%%%%%%%%%%%%%%%

\subsection{Notation}

Given a set $M$ and two functions $f_1,\, f_2:M\to\R$, we write $f_1(m) \lesssim f_2(m)$ if there is a numerical constant $c$ such that $f_1(m) \leq c\, f_2(m)$ for all $m\in M$. The symbol $\gtrsim$ is defined analogously. For any $p\in [1,\infty]$ and $u\in L^p(\Omega)$ we denote 
$$
\|u\|_p = \|u\|_{L^p(\Omega)}. 
$$
If $p = 2$, we typically drop the subscript and write $\|u\| = \|u\|_{L^2(\Omega)}$.

%%%%%%%%%%%%%%%%%%%%%%%%%%%%%%%%%%%%%%%%%%%%%%%%%%%%%%%%%%%

\section{\bf Upper bound on $S(a+\epsilon V)$}
\label{sec-upperb}

Recall that we always work under Assumption \ref{ass}. In this section (and only in this section), however, we do \emph{not} assume \eqref{eq:anegative}.

\subsection{Statement of the bounds and consequences}

Our goal in this section is to prove an upper bound on $S(a+\epsilon V)$ by evaluating the quotient $\mathcal S_{a+\epsilon V}[\cdot]$ on a certain family of trial functions. For $x\in\Omega$ and $\lambda>0$, let
\begin{equation}\label{eq:psi}
\psi_{x, \lambda}(y) := PU_{x, \lambda}(y) -  \lambda^{-1/2} (H_a(x, y) - H_0(x, y)) \,.
\end{equation}
This function belongs to $H^1_0(\Omega)$. We shall prove the following expansions.

\begin{thm}\label{exppsi}
As $\lambda\to\infty$, uniformly for $x$ in compact subsets of $\Omega$ and for $\epsilon\geq 0$,
\begin{align}
\label{eq:uppern}
\int_\Omega \left( |\nabla\psi_{x,\lambda}|^2 + (a+\epsilon V)\psi_{x,\lambda}^2\right)dy
& = 3 \left( \frac{S}{3}\right)^{\frac{3}{2}} - 4\pi\, \phi_a(x)\,\lambda^{-1}
+ 2\pi(4-\pi)\,a(x)\, \lambda^{-2} + \frac{\eps}{\lambda} Q_V(x) \nonumber \\
& \quad  + o(\lambda^{-2})  +o( \epsilon\lambda^{-1})
\end{align}
and
\begin{equation}
\label{eq:upperd}
\int_\Omega \psi_{x,\lambda}^6\,dy = \left( \frac{S}{3}\right)^{\frac{3}{2}} -8\pi \phi_a(x) \lambda^{-1} + 8\pi \, a(x)\, \lambda^{-2} +15\pi^2\, \phi_a(x)^2\,\lambda^{-2} + o(\lambda^{-2}) \,.
\end{equation}
In particular,
\begin{align}\label{eq:upperquo}
\mathcal S_{a+\epsilon V}[\psi_{x,\lambda}]  = S &+ \left( \frac{S}{3}\right)^{-\frac{1}{2}} 4\pi\,\phi_a(x)\,\lambda^{-1} \notag \\
&  + \left( \frac{S}{3}\right)^{-\frac{1}{2}}  \left( \frac{\eps}{\lambda} Q_V(x) - 2 \pi^2\, a(x)\, \lambda^{-2} - (15\pi^2 -128)\,\phi_a(x)^2\, \lambda^{-2} \right) \notag \\
&  + o(\lambda^{-2}) +o(\eps\lambda^{-1}) \,.
\end{align}
\end{thm}

In the proof of Theorem \ref{exppsi} we do not use the fact that $a$ is critical. We only use the fact that $-\Delta +a$ is coercive. In the following corollary we use criticality.

\begin{cor} \label{cor-2}
One has $\phi_a(x)\geq 0$ for all $x\in\Omega$ and $a(x)\leq 0$ for all $x\in\mathcal N_a$. 
\end{cor}

The first part of this corollary appears in \cite[Thm. 7]{Br}. Note that the second part is non-trivial since we do \emph{not} assume \eqref{eq:anegative}.

\begin{proof}
We apply \eqref{eq:upperquo} with $\epsilon=0$. We get $\mathcal S_a[\psi_{x,\lambda}] = S+(S/3)^{-1/2} 4\pi \phi_a(x) \lambda^{-1} + o(\lambda^{-1})$ for any fixed $x\in\Omega$. Since $S=S(a)\leq \mathcal S_a[\psi_{x,\lambda}]$, we infer that $\phi_a(x)\geq 0$ for all $x\in\Omega$. Similarly, $\mathcal S_a[\psi_{x,\lambda}] = S-(S/3)^{-1/2} 2\pi^2 a(x) \lambda^{-2} + o(\lambda^{-2})$ for any fixed $x\in\mathcal N_a$ implies that $a(x)\leq 0$ for all $x\in\mathcal N_a$.
\end{proof}

\begin{cor}\label{thm-upperb}
Assume that $\mathcal N_a(V)\neq\emptyset$. Then $S(a+\epsilon V)<S$ for all $\epsilon>0$ and, as $\epsilon\to 0+$,
$$
S(a+\epsilon V) \leq  
S - \left( \frac{S}{3}\right)^{-\frac{1}{2}} \frac{1}{8\pi^2} \sup_{x\in\mathcal N_a(V)} \frac{Q_V(x)^2}{|a(x)|} \, \epsilon^2 + o(\epsilon^{2})\, ,
  $$
where the right side is to be understood as $-\infty$ if $a(x)=0$ for some $x\in \mathcal N_a(V)$. 
\end{cor}

\begin{proof}
We fix $x\in\mathcal N_a$ and $k>0$ and apply \eqref{eq:upperquo} with $\lambda = (k\epsilon)^{-1}$. Since $S(a+\epsilon V) \leq \mathcal S_a[\psi_{x,\lambda}]$, we obtain
$$
\limsup_{\epsilon\to 0} \frac{S(a+\epsilon V) - S}{\epsilon^2} \leq (S/3)^{-1/2}  \left( k \int_\Omega  V\, G_a^2(x,y)\,dy - 2 \pi^2\, a(x)\, k^2 \right).
$$
Thus,
$$
\limsup_{\epsilon\to 0} \frac{S(a+\epsilon V) - S}{\epsilon^2} \leq (S/3)^{-1/2}  \inf_{x\in\mathcal N_a,\, k>0} \left( k \int_\Omega  V\, G_a^2(x,y)\,dy - 2 \pi^2\, a(x)\, k^2 \right),
$$
which implies the claimed upper bound.

For each $u\in H^1_0(\Omega)$, $\epsilon\mapsto \mathcal S_{a+\epsilon V}[u]$ is an affine linear function, and therefore its infimum over $u$, which is $\epsilon\mapsto S(a+\epsilon V)$, is concave. Since $S(a+\epsilon V)<S$ for all sufficiently small $\epsilon>0$, as we have just shown, we conclude that $S(a+\epsilon V)<S$ for all $\epsilon>0$.
\end{proof}

%%%%%%%%%%%%%%%%%

\subsection{Auxiliary facts}

In this preliminary subsection we collect some expansions that will be useful in the proof of Theorem \ref{exppsi} as well as later on. In order to emphasize that criticality is not needed, we state them for a function $b\in C(\overline\Omega)\cap C^1(\Omega)$ such that the operator $-\Delta +b$ in $\Omega$ with Dirichlet boundary conditions is coercive.

\begin{lem} \label{lem-V}
As $\lambda\to\infty$, uniformly in $x$ from compact subsets of $\Omega$,
\begin{align*}
\left\| ( U_{x,\lambda}   - \lambda^{-1/2} H_b(x,\cdot)) - \lambda^{-1/2} G_b(x,\cdot) \right\|_{6/5} & = \mathcal{O}(\lambda^{-2}) \,, \\
\left\| ( U_{x,\lambda}   - \lambda^{-1/2} H_b(x,\cdot))^2 - \lambda^{-1} G_b(x,\cdot)^2 \right\|_1 & = \mathcal{O}(\lambda^{-2}\ln\lambda) \,.
\end{align*}
\end{lem}

\begin{proof}
Since 
$$
( U_{x,\lambda}   - \lambda^{-1/2} H_b(x,y)) - \lambda^{-1/2} G_b(x,y) =
- \lambda^{-1/2} \left( \frac{1}{|x-y|} - \frac{\lambda}{\sqrt{1+\lambda^2|x-y|^2}} \right),
$$
the first bound follows immediately from
\begin{equation}
\label{u-est}
0\leq \frac{1}{|x-y|}- \frac{\lambda}{\sqrt{1+\lambda^2 |x-y|^2}} 
\leq 
\min\left\{ \frac{1}{|x-y|}, \frac{1}{2\lambda^2|x-y|^3} \right\}.
\end{equation}
To prove the second bound, we write
\begin{align*}
( U_{x,\lambda}   - \lambda^{-1/2} H_b(x,y))^2 - \lambda^{-1} G_b^2(x,y) & = - \lambda^{-1} \Big(\frac{1}{|x-y|^2}- \frac{\lambda^2}{1+\lambda^2 |x-y|^2 }\Big) \\
& \quad + 2  \lambda^{-1} H_b(x,y) \left(\frac{1}{|x-y|} - \frac{\lambda}{\sqrt{1+\lambda^2 |x-y|^2)} }  \right). %\label{V-aux1}
\end{align*}
The last term on the right side can be bounded as before, using the fact that $H_b(x,\cdot)$ is uniformly bounded in $L^\infty(\Omega)$ for $x$ in compact subsets of $\Omega$, see \eqref{sup-h-2} below. The first term on the right side can be bounded using
$$
0\leq \frac{1}{|x-y|^2}- \frac{\lambda^2}{1+\lambda^2 |x-y|^2 } 
\leq \min\left\{ \frac{1}{|x-y|^2}, \frac{1}{\lambda^2|x-y|^4} \right\} \, . 
$$
This proves the lemma.
\end{proof}

\begin{lem} \label{lem-uh}
As $\lambda\to \infty$, uniformly for $x$ in compact subsets of $\Omega$,
\begin{align*}
\int_\Omega U_{x,\lambda}^5\, H_b(x, y)\,dy & = \frac{4\pi}{3}\,   \phi_b(x) \, \lambda^{-1/2}-\frac{4\pi}{3}\, b(x) \, \lambda^{-3/2} +o( \lambda^{-3/2}) \,.
\end{align*}
\end{lem}

\begin{proof}
\emph{Step 1.}
We claim that, with $d(x):= \mathrm{dist}(x,\partial\Omega)$,
\begin{equation}\label{sup-h-2}
\| H_b(x, \cdot) \|_\infty \ \lesssim\ d(x)^{-1}
\qquad \text{for all}\ x\in\Omega \,.
\end{equation}

Indeed, since $H_0(x,\cdot)$ is harmonic in $\Omega$, the maximum principle implies
\begin{equation}
\label{sup-h}
\| H_0(x,\cdot)\|_\infty = \sup_{y\in\partial\Omega} H_0(x,y) = d(x)^{-1} \,.
\end{equation}
In order to deduce \eqref{sup-h-2} we note that the resolvent identity implies
\begin{equation}  \label{resolvent}
H_b(x,y) - H_0(x,y) = \frac{1}{4\pi} \int_\Omega G_0(x,z) b(z) G_b(z,y)\, dz\,.
\end{equation}
The claim now follows from the fact that
$$
\sup_{x,y\in\Omega} \int_\Omega  G_0(x,z)\, G_b(z,y)\, dz \ < \ \infty \,.
$$

\emph{Step 2.}
We claim that for any $x\in\Omega$ there is a $\xi_x\in\R^3$ such that 
\begin{equation} \label{taylor-h}
H_b(x,y) = H_b(x,x) +  \xi_x \cdot(y-x)  -\frac{b(x)}{2}\, |y-x| + o(|y-x|) \qquad \text{ as}  \quad  y\to x \,.
\end{equation}
The asymptotics are uniform for $x$ from compact subsets of $\Omega$.

To prove this, let
\begin{equation} \label{eq-psi}
\Psi_x(y) := H_b(x,y) -H_b(x,x) + \frac{b(x)}{2}\, |y-x| \,. 
\end{equation}
Using the equation
\begin{equation} \label{Ha-pde}
\Delta_y\, H_a(x,y) + a(y)\, G_a(x,y) = 0
\end{equation}
as well as the fact that $\Delta|x|=2|x|^{-1}$ as distributions we see that $\Psi_x$ is a distributional solution of 
\begin{equation} \label{eq-psi2}
-\Delta_y \Psi_x(y) =  F_x(y)    \qquad \text{in} \ \ \Omega,
\end{equation}
where 
$$
F_x(y) := \frac{b(y)-b(x)}{|x-y|} - b(y) H_b(x,y) \,. 
$$
By Step 1 and the assumption $b\in C(\overline\Omega)\cap C^1(\Omega)$, we have $F_x\in L^\infty_{\rm loc}(\Omega)$. In particular, $F_x\in L^p_{\rm loc}(\Omega)$ for any $3<p<\infty$ and therefore, by elliptic regularity (see, e.g., \cite[Thm. 10.2]{LiLo}), $\Psi_x\in C^{1,\alpha}_{\rm loc}(\Omega)$ for $\alpha = 1-3/p$. Thus, in particular, $\Psi_x\in C^1(\Omega)$. Inserting the Taylor expansion
$$
\Psi_x(y) = \nabla_y\Psi_x(x)\cdot (y-x) + o(|y-x|) \qquad \text{ as } \quad  y\to x
$$
into \eqref{eq-psi}, we obtain the claim with $\xi_x =  \nabla_y\Psi_x(x)$. The uniformity statement follows from the fact that if $x$ is from a compact set $K\subset\Omega$, then there is an open set $\omega$ with $K\subset\omega\subset\overline\omega\subset\Omega$ such that the norm of $F_x$ in $L^p(\omega)$ is uniformly bounded for $x\in K$.

\emph{Step 3.}
We now complete the proof of the lemma. Let $0<\rho\leq d(x)$ and write, using Step 2,
\begin{align*}
\int_\Omega U_{x,\lambda}^5 H_b(x,y)\,dy & = \phi_b(x) \int_{B_\rho(x)} U_{x,\lambda}^5 \,dy + \int_{B_\rho(x)} U_{x,\lambda}^5 \xi_x\cdot(y-x)\,dy - \frac{b(x)}{2} \int_{B_\rho(x)} U_{x,\lambda}^5 |y-x| \,dy \\
& \quad + o\left( \int_{B_\rho(x)} U_{x,\lambda}^5 |y-x| \,dy\right) + \int_{\Omega\setminus B_\rho(x)} U_{x,\lambda}^5 H_b(x,y)\,dy 
\end{align*}
with $\rho\to 0$ as $\lambda\to\infty$. Since $x$ belongs to a compact subset of $\Omega$, we have $d(x)\gtrsim 1$, and therefore the bound \eqref{sup-h-2} from Step 1 implies
\begin{align*}
\left| \int_{\Omega \setminus B_\rho(x)} U_{x,\lambda}^5\, H_a(x, y)\,dy \right| &\  \lesssim \ \int_{\Omega \setminus B_\rho(x)} U_{x,\lambda}^5\,dy \leq \lambda^{-1/2}\, 4\pi \int_{\lambda\rho}^\infty \frac{t^2\, dt}{(1+t^2)^{5/2}} =   \mathcal{O}\left(\lambda^{-5/2}\, \rho^{-2} \right).
\end{align*}
Similarly,
\begin{align*}
\int_{B_\rho(x)} U_{x,\lambda}^5 \, dy & = \lambda^{-1/2}\, 4\pi \int_0^{\lambda\rho} \frac{t^2\, dt}{(1+t^2)^{5/2}} = \frac{4\pi}{3}\,\lambda^{-1/2} + \mathcal{O}\left(\lambda^{-5/2} \rho^{-2} \right)
\end{align*}
and
\begin{align*}
\int_{B_\rho(x)} U_{x,\lambda}^5\, |x-y| \, dy & =  4\pi\,  \lambda^{-\frac 32} \left(\int_0^\infty   \frac{t^3\, dt}{(1+t^2)^{5/2}} - \int_{\rho\lambda}^\infty \frac{t^3\, dt}{(1+t^2)^{5/2}}\right) \\
& =   \frac{8\pi}{3}\, \lambda^{-\frac 32} +  \mathcal{O}\left( \lambda^{-5/2}\, \rho^{-1} \right).
\end{align*}
Finally, since $U_{x,\lambda}$ is radial about $x$,
\begin{equation} \label{mean-0} 
\int_{B_\rho(x)} U_{x,\lambda}^5(y) \, \xi_x \cdot (y-x)\, dy =0 \,.
\end{equation}
Choosing $\rho\to 0$ with $\lambda \rho^{2}\to\infty$ we obtain the conclusion of the lemma.
\end{proof}

The argument in Step 2 is the only place in this paper where we use the $C^1$ assumption on $a$. Clearly the same proof would work if we only assumed $a\in C^{1,\alpha}(\Omega)$ for some $\alpha>0$.

\begin{lem} \label{lem-uh2}
As $\lambda\to \infty$, uniformly for $x$ in compact subsets of $\Omega$,
\begin{align*}
\int_\Omega U_{x,\lambda}^4\, H_b(x, y)^2\,dy & = \pi^2\,   \phi_b(x)^2 \, \lambda^{-1} + o( \lambda^{-1}) \,.
\end{align*}
\end{lem}

The proof is similar, but simpler than that of Lemma \ref{lem-uh} and is omitted. We only note that the constant comes from
$$
\int_{\R^3} U_{x,\lambda}^4\,dy = 4\pi\,\lambda^{-1} \int_0^\infty \frac{t^2\,dt}{(1+t^2)^2} = \pi^2\, \lambda^{-1} \,. 
$$

\begin{lem} \label{lem-int-a}
As $x\to\infty$, uniformly for $x$ from compact subsets of $\Omega$,
\begin{equation*}
\int_\Omega b(y) \, U_{x,\lambda}(y) \left( \frac{\lambda^{-\frac 12}}{|x-y|} - U_{x,\lambda}(y)\right) dy = 2  \pi (\pi -2)\, b(x)\,\lambda^{-2} +   \mathcal{O}\left( \lambda^{-3}\log\lambda \right) \,.
\end{equation*}
\end{lem}

\begin{proof}
Let $0<\rho\leq\mathrm{dist}(x,\partial\Omega)$. Since $\frac{\lambda^{-\frac 12}}{ |x-y|} - U_{x,\lambda}(y) \geq 0$ for any $x,y\in\Omega$, the differentiability of $b$ at $x$ implies
\begin{align*}
 \int_{B_\rho(x)} b(y) \, U_{x,\lambda}(y) \left( \frac{\lambda^{-\frac 12}}{|x-y|} - U_{x,\lambda}(y)\right) dy  & =  b(x) \int_{B_\rho(x)}  U_{x,\lambda}(y) \left( \frac{\lambda^{-\frac 12}}{|x-y|} - U_{x,\lambda}(y)\right) dy  + R_\lambda
\end{align*}
with 
\begin{align} 
| R_\lambda |  & \ \lesssim \   \int_{B_\rho(x)} |x-y| \, U_{x,\lambda}(y) \left( \frac{\lambda^{-\frac 12}}{|x-y|} - U_{x,\lambda}(y)\right) dy \nonumber \\
& \ \lesssim \ \lambda^{-3}
\int_0^{\rho\lambda} \left(\frac{t^2}{\sqrt{1+t^2}} - \frac{t^3}{1+t^2}\right) dt =  \mathcal{O}\left( \lambda^{-3} \, \ln(\lambda \rho) \right).  \label{eq-R}
\end{align}
Moreover, 
\begin{align*}
\int_{B_\rho(x)}  U_{x,\lambda}(y) \left( \frac{\lambda^{-\frac 12}}{|x-y|} - U_{x,\lambda}(y)\right) dy & = \lambda^{-2} \, 4 \pi  \int_0^{\rho\lambda} \left(\frac{t}{\sqrt{1+t^2}} - \frac{t^2}{1+t^2}\right) dt \\
& = \lambda^{-2} \, 2\pi (\pi-2) \, (1+\mathcal{O}((\lambda\rho)^{-1})) \,. 
\end{align*}
On the complement of $B_\rho(x)$ we use the bound \eqref{u-est}, which gives
\begin{align*}
\left| \int_{\Omega\setminus B_\rho(x)} b(y) \, U_{x,\lambda}(y) \left( \frac{\lambda^{-\frac 12}}{|x-y|} - U_{x,\lambda}(y)\right) dy \right| 
 &  \lesssim \  \lambda^{-2} \int_{\rho\lambda}^\infty \frac{dt}{t\,(1+ t^2)^{1/2}}  = \mathcal{O}(\rho^{-1}\, \lambda^{-3}) \,.
\end{align*}
Choosing $\rho=1/\ln\lambda$ we obtain the bound in the lemma.
\end{proof}

The same proof shows that if $b$ is merely continuous, but not necessarily $C^1$, then the expansion still holds with an error $o(\lambda^{-2})$. This would be sufficient for our analysis.

%%%%%%%%%%%%%%%%%

\subsection{Expansion of the numerator}

One easily checks that for all $x\in\R^3$ and $\lambda>0$,
\begin{equation}  \label{el}
-\Delta U_{x,\lambda} = 3\, U_{x,\lambda}^5 \,.
\end{equation}
This, together with the equation \eqref{Ha-pde}, the harmonicity of $H_0(x,\cdot)$ and \eqref{eq-pu}, implies that 
\begin{equation}
\label{eq:eqpsi}
-\Delta_y \psi_{x, \lambda}(y) =  -\Delta_y U_{x,\lambda}(y) + \lambda^{-\frac 12}\, \Delta_y H_a(x,y) = 3\, U_{x,\lambda}^5(y) -\lambda^{-\frac 12}\, a(y)\, G_a(x,y).
\end{equation}
We now introduce $f_{x,\lambda}$ by
\begin{equation}
\label{PU exp} 
PU_{x,\lambda}= U_{x,\lambda}  - \lambda^{-1/2} H_0(x,\cdot ) - f_{x,\lambda}\, ,
\end{equation}
and recall that \cite[Prop. 1 (b)]{rey2}, with $d:=\mathrm{dist}(x,\partial\Omega)$,
\begin{equation} \label{sup-f} 
\|f_{x, \lambda}\|_\infty = \mathcal{O}(\lambda^{-5/2} d^{-3}) \,.
\end{equation}
Hence, by \eqref{eq:eqpsi} and the fact that $\psi_{x,\lambda}$ vanishes on the boundary,
\begin{align}\label{eq:expnp1}
\int_\Omega | \nabla \psi_{x, \lambda}|^2 & = \int_\Omega \left(3\, U_{x,\lambda}^5(y) -\lambda^{-\frac 12}\, a(y)\, G_a(x,y)\right) \left( U_{x,\lambda}(y) -\lambda^{-\frac 12}\, H_a(x,y) - f_{x,\lambda}(y)\right) dy \nonumber \\
& = 3 \int_\Omega U_{x,\lambda}^6(y)\, dy - 3 \,  \lambda^{-\frac 12}  \int_\Omega U_{x,\lambda}^5(y)\, H_a(x,y)\, dy \nonumber \\
& \quad - \lambda^{-1/2} \int_\Omega a(y)\, G_a(x,y)\, \left(U_{x,\lambda}(y) -\lambda^{-\frac 12}\, H_a(x,y) \right) dy 
\nonumber \\
& \quad - \int_\Omega \left(3\, U_{x,\lambda}^5(y) -\lambda^{-\frac 12}\, a(y)\, G_a(x,y)\right) f_{x,\lambda}(y)\,dy \,.
\end{align}
It is easy to see that
$$
\int_\Omega \left| 3\, U_{x,\lambda}^5(y) -\lambda^{-\frac 12}\, a(y)\, G_a(x,y) \right|dy = \mathcal O(\lambda^{-1/2})
$$
and therefore, by \eqref{sup-f} and the fact that $x$ is in a compact subset of $\Omega$,
$$
\int_\Omega \left(3\, U_{x,\lambda}^5(y) -\lambda^{-\frac 12}\, a(y)\, G_a(x,y)\right) f_{x,\lambda}(y)\,dy = \mathcal O(\lambda^{-3}) \,.
$$
A simple computation shows that the first term on the right side of \eqref{eq:expnp1} is
\begin{equation} \label{uq}
\int_\Omega  U_{x, \lambda}^6\,dy = \int_{\R^n}  U_{x, \lambda}^6\,dy  + \mathcal{O}(\lambda^{-3}) = \left( \frac{S}{3}\right)^{\frac{3}{2}} + \mathcal{O}(\lambda^{-3}) \,.
\end{equation}
For the second term we use Lemma \ref{lem-uh} and obtain
$$
3 \,  \lambda^{-\frac 12}  \int_\Omega U_{x,\lambda}^5(y)\, H_a(x,y)\, dy
= 4\pi \phi_a(x) \lambda^{-1} - 4\pi a(x)\lambda^{-2} + o(\lambda^{-2}) \,.
$$
We will combine the third term with the term coming from $\int_\Omega a\psi_{x,\lambda}^2\,dy$.

Using again expansion \eqref{PU exp} of $PU_{x,\lambda}$ we find
\begin{align*}
\int_\Omega (a+\epsilon V) \psi_{x, \lambda}^2(y)\, dy & = \int_\Omega (a+\epsilon V) \left(U_{x,\lambda} -\lambda^{-1/2}\, H_a(x,y)\right)^2 dy \\
& \quad - 2\int_\Omega (a+\epsilon V) (U_{x,\lambda}-\lambda^{-1/2} H_a(x,y))f_{x,\lambda}\,dy + \int_\Omega (a+\epsilon V) f_{x,\lambda}^2\,dy \,.
\end{align*}
Using \eqref{sup-f} and the fact that $x$ is in a compact subset of $\Omega$ it is easy to see that
\begin{align*}
- 2\int_\Omega (a+\epsilon V) (U_{x,\lambda}-\lambda^{-1/2} H_a(x,y))f_{x,\lambda}\,dy + \int_\Omega (a+\epsilon V) f_{x,\lambda}^2\,dy = \mathcal O(\lambda^{-3} (1+\epsilon)) \,.
\end{align*}

To summarize, we have shown that
\begin{align*} 
\int_\Omega \left( | \nabla \psi_{x, \lambda}|^2 + a\, \psi_{x, \lambda}^2\right)dy & = 3 \left( \frac{S}{3}\right)^{\frac{3}{2}}  - 4\pi\, \phi_a(x)\,\lambda^{-1} + 4\pi\,a(x)\,\lambda^{-2} + T(x,\lambda) \notag \\
& \quad + \epsilon \int_\Omega V (U_{x,\lambda}-\lambda^{-1/2} H_a(x,y))^2\,dy + o(\lambda^{-2}) + \mathcal O(\epsilon\lambda^{-3} ) %\label{kin+a}
\end{align*}
with
\begin{align*}
T(x,\lambda) := \int_\Omega a(y)\, \left(U_{x,\lambda}(y) -\lambda^{-1/2}\, H_a(x,y) \right) \left(U_{x,\lambda}(y)- \frac{\lambda^{-1/2}}{|x-y|} \right) dy \,.
\end{align*}
Similarly as in the proof of Lemma \ref{lem-int-a} one finds that
$$
\lambda^{-1/2} \int_\Omega a(y) \, H_a(x,y)  \Big (\frac{\lambda^{-1/2}}{|x-y|} - U_{x,\lambda}(y) \Big ) dy = \mathcal{O}(\lambda^{-3}\ln\lambda) \,.
$$
Hence, by Lemma \ref{lem-int-a},
\begin{equation*} %\label{eq-T}
T(x,\lambda) = - 2\pi(\pi -2)\,a(x) \, \lambda^{-2} +  o(\lambda^{-2}). 
\end{equation*}

Finally, by Lemma \ref{lem-V},
$$
\int_\Omega V (U_{x,\lambda} - \lambda^{-1/2} H_a(x,y))^2 \,dy = \lambda^{-1} \int_\Omega V G_a(x,y)^2\,dy + \mathcal O(\lambda^{-2}\ln\lambda) \,.
$$
This proves the first assertion in Theorem \ref{exppsi}.

%%%%%%%%%%%%

\subsection{Expansion of the denominator}

By the decomposition \eqref{PU exp} for $PU_{x,\lambda}$ we obtain
$$
\int_\Omega  \psi_{x, \lambda}^6\,dy = \int_\Omega (U_{x,\lambda} - \lambda^{-1/2} H_a(x,y))^6\,dy + \mathcal O( \| U_{x,\lambda}- \lambda^{-1/2} H_a(x,\cdot)\|_5^5 \|f_{x,\lambda}\|_\infty + \|f_{x,\lambda}\|_6^6).
$$
Using \eqref{sup-h-2} and \eqref{sup-f}, together with the fact that $x$ is in a compact subset of $\Omega$, we see that the remainder term is $\mathcal O(\lambda^{-3})$. Next, we expand
\begin{align*}
\int_\Omega (U_{x,\lambda} - \lambda^{-1/2} H_a(x,y))^6\,dy
& = \int_\Omega U_{x,\lambda}^6\,dy - 6 \lambda^{-1/2} \int_\Omega U_{x,\lambda}^5  H_a(x,y)\,dy + 15 \lambda^{-1} \int_\Omega U_{x,\lambda}^4  H_a(x,y)^2 \,dy \\
& \quad + \mathcal O(\lambda^{-3/2} \|U_{x,\lambda}\|_3^3 \|H_a(x,\cdot)\|_\infty^2 + \lambda^{-3} \| H_a(x,\cdot)\|_6^6 ) \,.
\end{align*}
Using \eqref{sup-h-2}, together with the fact that $x$ is in a compact subset of $\Omega$, we see that the remainder term is $\mathcal O(\lambda^{-3}\ln\lambda)$. The first three terms on the right side are evaluated in \eqref{uq} and Lemmas \ref{lem-uh} and \ref{lem-uh2}. This proves the second assertion in Theorem \ref{exppsi}.

%%%%%%%%%%

\subsection{Expansion of the quotient}

Expansion \eqref{eq:upperd} implies that
\begin{align*}
\left( \int_\Omega \psi_{x,\lambda}^6\,dy \right)^{-1/3} & =  \left( \frac{S}{3} \right)^{-\frac12} + \left( \frac{S}{3} \right)^{-2} \frac{8\pi}{3}\,\phi_a(x)\, \lambda^{-1} \\
& \quad + \left( \frac{S}{3} \right)^{-2} \left( - \frac{8\pi}{3}\, a(x) - 5\pi^2\,\phi_a(x)^2 + \frac{2}{9} \frac{64\pi^2}{(S/3)^{3/2}}\,\phi_a(x)^2 \right) \lambda^{-2} + o(\lambda^{-2}).
\end{align*}
Expansion \eqref{eq:upperquo} now follows by multiplying the previous equation with \eqref{eq:uppern}. This concludes the proof of Theorem \ref{exppsi}.

%%%%%%%%%%%%%%%%%%%%%%%%%%%%%%%%%%%%%%%%%%%%%%%%%%%%%%%%%

\section{\bf Lower bound on $S(a+\epsilon V)$. Preliminaries} \label{sec-lb}

\subsection{The asymptotic form of almost minimizers}

The remainder of this paper is concerned with proving a lower bound on $S(a+\epsilon V)$ which matches the upper bound from Corollary \ref{thm-upperb}. We will establish this by proving that functions $u_\epsilon$ for which $\mathcal S_{a+\epsilon V}[u_\epsilon]$ is `close' to $S(a+\epsilon V)$ are `close' to the functions $\psi_{x,\lambda}$ used in the upper bound for certain $x$ and $\lambda$ depending on $\epsilon$. We will prove this in several steps. The very first step is the following proposition. 

\begin{prop} \label{prop-app-min}
Let $(u_\epsilon)\subset H^1_0(\Omega)$ be a sequence of functions  satisfying
\begin{equation}
\label{eq:appminass}
\mathcal S_{a+\epsilon V}[u_\epsilon] = S+o(1) \,,
\qquad
\int_\Omega u_\epsilon^6\,dx = (S/3)^{3/2} \,.
\end{equation}
Then, along a subsequence,
\begin{equation} \label{u-rey}
u_\epsilon = \alpha_\epsilon\left( PU_{x_\epsilon,\lambda_\epsilon} + w_\epsilon\right),
\end{equation}
where
\begin{equation}
\begin{aligned}  \label{lim-k}
\alpha_\epsilon & \to s \qquad \text{for some}\ s\in\{-1,+1\} \,,\\
x_\epsilon & \to x_0 \qquad\text{for some}\ x_0\in\overline\Omega \,, \\
\lambda_\epsilon d_\epsilon & \to \infty \,,\\
\| \nabla w_\epsilon \| &\to 0 \qquad\text{and}\qquad w_\epsilon \in T_{x_\epsilon,\lambda_\epsilon}^\bot \,.  
\end{aligned}
\end{equation}
Here $d_\eps=$dist$(x_\eps,\partial\Omega)$. 
\end{prop}

If the $u_\epsilon$ are minimizers for $S(a+\epsilon V)$, and therefore solutions to the corresponding Euler--Lagrange equation, this proposition is well-known and goes back to work of Struwe \cite{St} and Bahri--Coron \cite{BaCo}. The result for almost minimizers is also well-known to specialists, but since we have not been able to find a proof in the literature, we include one in Appendix \ref{sec-app-b}. Here we only emphasize that the fact that $u_\epsilon$ converges weakly to zero in $H^1_0(\Omega)$ is deduced from a theorem of Druet \cite{dr} which says that $S(a)$ is not attained for critical $a$. (Note that this part of the paper \cite{dr} is valid for $a\in L^{3/2}(\Omega)$, without any further regularity requirement.)

\subsection*{Convention}
From now on we will assume that
\begin{equation}
\label{appr-min0}
S(a+\epsilon V)< S
\qquad\text{for all}\ \epsilon>0
\end{equation}
and that $(u_\epsilon)$ satisfies \eqref{appr-min}. In particular, assumption \eqref{eq:appminass} is satisfied. We will always work with a sequence of $\epsilon$'s for which the conclusions of Proposition \ref{prop-app-min} hold. To enhance readability, we will drop the index $\epsilon$ from $\alpha_\epsilon$, $x_\epsilon$, $\lambda_\epsilon$, $d_\epsilon$ and $w_\epsilon$. 

%%%%%%%%%%%%%%%%%%%%%%%%%%%%%%%%%%%%%%%%%%%%%%%%%%%%%%%%%%

\section{\bf A priori bounds}\label{sec:apriori}

%%%%%%%%%%%%%%%%%%%%%%%%%%%%%%%%%%%%%%%%%%%%

\subsection{Statement of the bounds}

From Proposition \ref{prop-app-min} we know that $\|\nabla w\|=o(1)$ and that the limit point $x_0$ of $(x_\epsilon)$ lies in $\overline\Omega$. The following proposition, which is the main result of this section, improves both these results.

\begin{prop}\label{wbound}
As $\epsilon\to 0$,
\begin{equation} \label{esp-w}
\| \nabla w\| =  \mathcal{O}\left(\lambda^{-1/2} \right) \,,
\end{equation}
\begin{equation} \label{eq-d}
d^{-1} = \mathcal{O}(1)
\end{equation}
and
\begin{equation}
\label{eq:energybound}
\lambda \left( S - S(a+\epsilon V) \right) = \mathcal O(1)
\qquad\text{and}\qquad
\lambda \left( \mathcal S_{a+\epsilon V}[u_\epsilon] - S(a+\epsilon V) \right) = o(1) \,.
\end{equation}
\end{prop}

The bounds \eqref{esp-w} and \eqref{eq-d} were shown in \cite[Lem.~2.2 and Thm.~1.1]{esp} in the case where $u_\epsilon$ is a minimizer for $S(a+\epsilon V)$. Since the proof in \cite{esp} uses the Euler--Lagrange equation satisfied by minimizers, this proof is not applicable in our case. We will replace the use of the Euler--Lagrange equation by a suitable expansion of $\mathcal S_{a+\epsilon V}[u_\epsilon]$, which is carried out in Subsection \ref{sec:expini}. The other ingredient in the proof of \cite[Lem.~2.2]{esp} and in our proof is the coercivity of a certain quadratic form, see Lemma \ref{coercivity} in Subsection \ref{sec:coercivity}. Finally, in Subsection \ref{sec:wboundproof} we will prove Proposition \ref{wbound}.

%%%%%%%%%%%%

\subsection{A first expansion}\label{sec:expini}

In this subsection, we shall prove the following lemma.

\begin{lem}\label{expini}
As $\epsilon\to 0$,
\begin{align*}
\mathcal S_{a+\epsilon V}[u_\eps]
& = S + (S/3)^{-1/2} 4\pi \phi_0(x)\lambda^{-1} + (S/3)^{-1/2} \int_\Omega (|\nabla w|^2 + a w^2 - 15\, U_{x,\lambda}^4 w^2)\,dy \\ 
& \quad + O\left(\lambda^{-1/2} \|\nabla w\|\right) + o((d\lambda)^{-1}) + o(\|\nabla w\|^2) \,.
\end{align*}
\end{lem}

\begin{proof}[Proof of Lemma \ref{expini}]
We will expand separately the numerator and the denominator in $\mathcal S_{a+\epsilon V}[u_\eps]$.

\medskip

\emph{Expansion of the numerator.}
Since $w$ is orthogonal to $PU$, we have
\begin{equation}
\label{eq:expnumortho}
\alpha^{-2} \int_\Omega |\nabla u_\epsilon|^2\,dy = \int_\Omega |\nabla PU_{x,\lambda}|^2\,dy + \int_\Omega |\nabla w|^2\,dy \,.
\end{equation}
The first term on the right side is computed in \eqref{eq:nablapu}. The other terms in the numerator are
$$
\alpha^{-2} \int_\Omega (a+\epsilon V) u_\epsilon^2\,dy = \int_\Omega (a+\epsilon V)PU_{\lambda, x}^2 \,dy + 2 \int_\Omega (a+\eps) PU_{\lambda, x} w \,dy+ \int_\Omega (a+\epsilon V) w^2\,dy \,.
$$
Since $0\leq PU_{x,\lambda}\leq U_{x,\lambda}\leq\lambda^{-1/2} |x-y|$, see \cite[Prop. 1]{rey2}, we have
$$
\left| \int_\Omega (a+\epsilon V)PU_{x,\lambda}^2 \,dy \right| \leq \|a + \epsilon V\|_\infty \lambda^{-1} \int_\Omega \frac{dy}{|x-y|^2} =\mathcal O(\lambda^{-1}) \,.
$$
Clearly,
$$
\epsilon\left| \int_\Omega V w^2\,dy \right| \leq \epsilon \|V\|_\infty \|w\|^2 \lesssim \epsilon \|V\|_\infty \|\nabla w\|^2 = o(\|\nabla w\|^2) \,,
$$
and, by \eqref{eq:zm65},
$$
\left| \int_\Omega (a+\epsilon V)PU_{x,\lambda}w \,dx \right| \leq \|a+\epsilon V\|_\infty \|PU_{x,\lambda}\|_{6/5} \|w\|_6 = \mathcal O(\lambda^{-1/2} \|\nabla w\|) \,.
$$
To summarize, the numerator is $\alpha^2$ times
$$
3^{-1/2} S^{3/2} - 4\pi \phi_0(x)\lambda^{-1} + \int_\Omega \left( |\nabla w|^2 + a w^2\right)dy +\mathcal O\left(\lambda^{-1/2} \|\nabla w\| \right) + o((\lambda d)^{-1}) + o( \|\nabla w\|^2) \,.
$$

\emph{Expansion of the denominator.}
We have
$$
\alpha^{-6} \int_\Omega u_\epsilon^6\,dy = \int_\Omega PU_{x,\lambda}^6\,dy + 6 \int_\Omega PU_{x,\lambda}^5 w\,dy + 15 \int_\Omega PU_{x,\lambda}^4 w^2\,dy + \mathcal O(\|\nabla w\|^3) \,.
$$
The first term on the right side is computed in \eqref{eq:pu6}. Moreover, abbreviating $\phi_{x,\lambda}:= \lambda^{-1/2} H_0(x,\cdot)+f_{x,\lambda}$, so that, by \eqref{PU exp}, $PU_{x,\lambda} = U_{x,\lambda}-\phi_{x,\lambda}$, we find
$$
\int_\Omega PU_{x,\lambda}^5 w \,dy = \int_\Omega U_{x,\lambda}^5 w \,dy + \mathcal O \left( \int_\Omega U_{x,\lambda}^4 \phi_{x,\lambda} |w| \,dy + \int_\Omega \phi_{x,\lambda}^5 |w|\,dy \right).
$$
(Note that $\phi_{x,\lambda}\geq 0$, since $PU_{x,\lambda}\leq U_{x,\lambda}$ by \cite[Prop.\! 1 (a)]{rey2}.) By \eqref{el}, \eqref{eq-pu}, the fact that $w$ vanishes on the boundary and since $w\in T_{x,\lambda}^\bot$, we have
$$
\int_\Omega U_{x,\lambda}^5 w\,dy = \frac13 \int_\Omega (-\Delta U_{x,\lambda}) w\,dy = \frac13 \int_\Omega \nabla PU_{x,\lambda}\cdot\nabla w\,dy = 0 \,.
$$
Also, by the equation after \cite[(10)]{esp},
$$
\int_\Omega U_{x,\lambda}^4 \phi_{x,\lambda} |w| \,dy + \int_\Omega \phi_{x,\lambda}^5 |w|\,dy = \mathcal O( (d\lambda)^{-1} \|\nabla w\|) =o((d\lambda)^{-1}) \,.
$$
Finally,
$$
\int_\Omega PU_{x,\lambda}^4 w^2 \,dy = \int_\Omega U_{x,\lambda}^4 w^2 \,dy + \mathcal O\left( \int_\Omega U_{x,\lambda}^3 \phi_{x,\lambda} w^2 \,dy + \int_\Omega \phi_{x,\lambda}^4 w^2\,dy \right)
$$
and, since $\|\phi_{x,\lambda}\|_6 = \mathcal O((d\lambda)^{-1/2})$ by \cite[Prop.\,1 (c)]{rey2},
$$
\int_\Omega U_{x,\lambda}^3 \phi_{x,\lambda} w^2 \,dy+ \int_\Omega \phi_{x,\lambda}^4 w^2\,dy = o(\|\nabla w\|^2) \,.
$$
To summarize, we have shown that
\begin{align*}
\alpha^{-6} \int_\Omega u_\epsilon^6 \,dy & = (S/3)^{3/2} -8\pi \phi_0(x)\lambda^{-1} +  15 \int_\Omega U_{x,\lambda}^4 w^2\,dy + o((d\lambda)^{-1})+ o(\|\nabla w\|^2)
\end{align*}
and therefore, by the rough bound $\int_\Omega U_{x,\lambda} w^2\,dy \leq \|U_{x,\lambda}\|_6^4 \|w\|_6^2 \lesssim \|U_{x,\lambda}\|_6^4 \|\nabla w\|^2 = o(1)$,
\begin{align*}
\alpha^2 \left(\int_\Omega u_\epsilon^6 \,dy\right)^{-1/3} & = \left( \frac S3\right)^{-\frac 12} + \left( \frac S3\right)^{-2} \frac{8\pi}{3} \phi_0(x)\lambda^{-1}
 - 45  S^{-2} \int_\Omega U_{x,\lambda}^4 w^2\,dy \\
 & \quad + o((d\lambda)^{-1}) + o(\|\nabla w\|^2) \,.
\end{align*}

The lemma follows immediately from the expansions of the numerator and the denominator.
\end{proof}

%%%%%%%%%%%%%%

\subsection{Coercivity}\label{sec:coercivity}

We will frequently use the following bound from \cite[Lem.~2.2]{esp}. 

\begin{lem}\label{coercivity}
There are constants  $T_*<\infty$ and $\rho>0$ such that for all $x\in\Omega$, all $\lambda>0$ with $d\lambda\geq T_*$ and all $v\in T_{x,\lambda}^\bot$,
\begin{equation} \label{eq-gapesposito}
\int_\Omega \left( |\nabla v|^2 + av^2 - 15\, U_{x,\lambda}^4 v^2\right)dy \geq \rho \int_\Omega |\nabla v|^2\,dy \,.
\end{equation}
\end{lem}

The proof proceeds by compactness, using the inequality \cite[(D.1)]{rey2}
\begin{equation*}
\int_\Omega \left( |\nabla v|^2 - 15\, U_{x,\lambda}^4 v^2 \right)dy \geq \frac47 \int_\Omega |\nabla v|^2\,dy
\qquad\text{for all}\ v\in T_{x,\lambda}^\bot \,.
\end{equation*}
For details of the proof we refer to \cite{esp}.

%%%%%%%%%%%%%%%%%%%%%%%%%%%%%%%%

\subsection{Proof of Proposition \ref{wbound}}\label{sec:wboundproof}

We combine the expansion from Lemma \ref{expini} with the coercivity bound from Lemma \ref{coercivity} and the fact that $c:=\inf_{y\in\Omega} \mathrm{dist}(y,\partial\Omega) \phi_0(y)>0$, see \cite[(2.8)]{rey2} or \cite[Lem.\! 8.3]{Fl}. (Note that this bound uses the $C^2$ assumption on $\partial\Omega$.) Thus,
\begin{align*}
\mathcal S_{a+\epsilon V}[u_\epsilon] \geq S + \left( (S/3)^{-1/2} 4\pi c + o(1) \right)(d\lambda)^{-1} + \left( (S/3)^{-1/2} \rho + o(1) \right) \|\nabla w\|^2 + \mathcal O(\lambda^{-1/2} \|\nabla w\|).
\end{align*}
Since $\lambda^{-1/2}\|\nabla w\|\leq \delta \|\nabla w\|^2 + (4\delta)^{-1} \lambda^{-1}$ for every $\delta>0$, we obtain, for all sufficiently small $\epsilon>0$ and some constants $c_1,c_2>0$ and $C<\infty$ independent of $\epsilon$,
$$
C\lambda^{-1} + \left( \mathcal S_{a+\epsilon V}[u_\epsilon] - S(a+\epsilon V) \right)
\geq S- S(a+\epsilon V) + c_1 (d\lambda)^{-1} + c_2 \|\nabla w\|^2 \,.
$$
By assumption \eqref{appr-min}, this becomes
$$
C\lambda^{-1} \geq (1+ o(1)) \left( S- S(a+\epsilon V)\right) + c_1 (d\lambda)^{-1} + c_2 \|\nabla w\|^2 \,.
$$
Since all three terms on the right side are non-negative, we obtain \eqref{esp-w}, \eqref{eq-d} and the first bound in \eqref{eq:energybound}. The second bound in \eqref{eq:energybound} follows from the first one by assumption \eqref{appr-min}. This completes the proof of the proposition.

\section{\bf A priori bounds reloaded}
\label{sec-reloaded}

\subsection{Statement and heuristics for the improved a priori bound}

In order to prove a sufficiently precise lower bound on $S(a+\epsilon V)$ we need more detailed information on
the almost minimizers $u_\eps$. Here we extract the leading term from the remainder term $w = w_\eps$ in \eqref{u-rey}. 

\begin{prop}  \label{thm-q}
One has, as $\epsilon\to 0$,
\begin{equation}
\label{eq:concpoint}
\lambda(S-S(a+\epsilon V)) = o(1) \,,
\qquad
\phi_a(x) = o(1)
\end{equation}
and
\begin{equation} \label{w-fkk}
w = - \lambda^{-1/2} (H_a(x, \cdot) - H_0(x, \cdot)) + q
\qquad\text{with}\qquad 
\|\nabla q\| = o(\lambda^{-1/2}) \,.
\end{equation}
\end{prop}

Note that the second statement in \eqref{eq:concpoint} implies that $\phi_a(x_0)=0$ for the limit point $x_0$ in \eqref{lim-k}. In particular, together with Corollary \ref{cor-2}, we obtain $\min_\Omega \phi_a = 0$ for critical $a$, which is Druet's theorem \cite{dr}. Our proof, which is closely related to that by Esposito \cite{esp}, uses another theorem of Druet, which says that $S(a)$ is not attained for critical $a$ \cite[Step 1]{dr} (see Proposition \ref{prop-app-min}), but is otherwise independent of \cite{dr}.

The proof of Proposition \ref{thm-q} is given at the end of this section. Let us explain the heuristics behind the proof. In Lemma \ref{expint} we will derive the following expansion,
\begin{align} \label{2-Salt2} 
\mathcal S_{a+\epsilon V}[u_\eps] & = S + \lambda^{-1} \left(\frac S3\right)^{-\frac12} \left( 4\pi  \, \phi_a(x) + (4\pi)^{-1} \iint_{\Omega\times\Omega} G_0(x,y) a(y) G_a(y,y') a(y')G_0(y',x)\,dy\,dy' \right) \nonumber  \\
& \quad + \left(\frac S3\right)^{-\frac12} \int_\Omega \left( |\nabla w|^2 + aw^2 + 2 \lambda^{-1/2} a G_0(x,y)w - 15\, U_{x,\lambda}^4 w^2 \right) dy + o(\lambda^{-1})   \,.
\end{align}
Note that this is an improvement over the expansion in Lemma \ref{expini}, which only had a remainder $\mathcal O(\lambda^{-1})$. This improvement is possible thanks to the information from Proposition \ref{wbound}.

From the expansion \eqref{2-Salt2} we want to determine the asymptotic form of $w$. In order to (almost) minimize the quotient $\mathcal S_{a+\epsilon V}[u_\eps]$ the function $w$ will (almost) minimize the expression
$$
\int_\Omega \left( |\nabla w|^2 + aw^2 + 2 \lambda^{-1/2} a G_0(x,y)w - 15\, U_{x,\lambda}^4 w^2 \right) dy \,.
$$
This is quadratic and linear in $w$, so it can be minimized by `completing a square'. If the term $-15\, U_{x,\lambda}^4$ were absent, then the minimum would be 
$$
-\lambda^{-1} (4\pi)^{-1} \iint_{\Omega\times\Omega}   G_0(x,y) a(y) G_a(y,y') a(y')G_0(y',x)\,dy\,dy'
$$
and the optimal choice for $w$ would be $-\lambda^{-1/2} (H_a(x, \cdot) - H_0(x, \cdot))$. Using the positive contribution that arises when completing the square, we will be able to show that if $u_\epsilon$ almost minimizes $S(a+\epsilon V)$, then $w$ almost minimizes the above problem and is therefore almost equal to $-\lambda^{-1/2} (H_a(x, \cdot) - H_0(x, \cdot))$. Proposition \ref{thm-q} provides a quantitative version of these heuristics.

As the above argument shows, the main difficulty will be to show that the term $-15\, U_{x,\lambda}^4$ is negligible to within $o(\lambda^{-1})$. This does not follow from a straightforward bound since $\|\nabla w\|^2$ is only $\mathcal O(\lambda^{-1})$. The orthogonality conditions satisfied by $w$ will play an important role.

%%%%%%%%%%%%%

\subsection{A second expansion}

In this subsection, we shall prove the following lemma.

\begin{lem}\label{expint}
As $\epsilon\to 0$,
\begin{align} \label{2-S} 
\mathcal S_{a+\epsilon V}[u_\eps] & = S + \lambda^{-1} \left(\frac S3\right)^{-\frac12} \left( 4\pi  \, \phi_a(x) + (4\pi)^{-1} \iint_{\Omega\times\Omega} G_0(x,y) a(y) G_a(y,y') a(y')G_0(y',x)\,dy\,dy' \right) \nonumber  \\
& \quad + \left(\frac S3\right)^{-\frac12} \int_\Omega \left( |\nabla w|^2 + aw^2 + 2 \lambda^{-1/2} a G_0(x,y)w - 15\, U_{x,\lambda}^4 w^2 \right) dy + o(\lambda^{-1})   \,.
\end{align} 
\end{lem}

\begin{proof}
\emph{Expansion of the numerator.}
We claim that
\begin{align} \label{1-num}
\alpha^{-2} \int_\Omega  (|\nabla u_\eps|^2 + a u_\epsilon^2 + \epsilon V u_\epsilon^2)\,dy &=  3^{-1/2} S^{3/2}  - \lambda^{-1} \left( 4\pi \phi_0(x) -  \int_\Omega a G_0(x,y)^2\,dy \right) \nonumber \\
& \quad + \int_\Omega \left(|\nabla w|^2 + a w^2 + 2 \lambda^{-1/2} \, a\, G_0(x,y) w\right)dy + o(\lambda^{-1}) \,. 
\end{align}

Indeed, arguing as in the proof of Lemma \ref{expini} and using the bounds on $d$ and $\|\nabla w\|$ from Proposition \ref{wbound}, we obtain
\begin{align*}
\alpha^{-2} \int_\Omega  (|\nabla u_\eps|^2 + a u_\epsilon^2 + \epsilon V u_\epsilon^2)\,dy &=  3^{-1/2} S^{3/2}  - 4\pi \phi_0(x) \lambda^{-1}  + \int_\Omega a PU_{x,\lambda}^2\,dy \nonumber \\
& \quad + \int_\Omega \left(|\nabla w|^2 + a w^2 + 2 \, a\, PU_{x,\lambda} w\right)dy + o(\lambda^{-1}) \,. 
\end{align*}
Note that here we have kept the term $\int_\Omega a (PU_{x,\lambda}^2 +2 PU_{x,\lambda} w)\,dy$ instead of estimating it. We now treat this contribution more carefully. We expand $PU_{x,\lambda}$ as in \eqref{PU exp}, which leads to
\begin{align*}
\int_\Omega a (PU_{x,\lambda}^2+ 2 PU_{x,\lambda} w)\,dy & = \int_\Omega a \left((U_{x,\lambda} -\lambda^{-1/2} H_0(x,y))^2 + 2(U_{x,\lambda}-\lambda^{-1/2} H_0(x,y))w\right)dy\\
& \quad -2 \int_\Omega a (PU_{x,\lambda}+w) f_{x,\lambda}\,dy - \int_\Omega a f_{x,\lambda}^2\,dy \,.
\end{align*}
By \eqref{sup-f} and \eqref{eq:zm65}, taking into account \eqref{eq-d},
$$
\left| \int_\Omega a \left( 2(PU_{x,\lambda}+w) f_{x,\lambda} + f_{x,\lambda}^2 \right)dy \right| = \mathcal O \left(\|a\|_\infty (\|PU_{x,\lambda}\|_{6/5}\|f\|_6 + \|w\|_6 \|f_{x,\lambda}\|_{6/5} + \|f_{x,\lambda}\|^2) \right) = \mathcal O(\lambda^{-3}) \,.
$$
On the other hand, by Lemma \ref{lem-V},
\begin{align*}
& \int_\Omega a \left((U_{x,\lambda} -\lambda^{-1/2} H_0(x,y))^2 +2 (U_{x,\lambda}-\lambda^{-1/2} H_0(x,y))w\right)dy \\
& \quad = \int_\Omega a \left( \lambda^{-1} G_0(x,y)^2 + 2\lambda^{-1/2} G_0(x,y)w \right)dy + \mathcal O(\lambda^{-2}\ln\lambda) \,.
\end{align*}
This proves \eqref{1-num}.

\emph{Expansion of the denominator.}
Combining the bound from the proof of Lemma \ref{expini} with the bounds on $d$ and $\|\nabla w\|$ from Proposition \ref{wbound}, we obtain
\begin{align}\label{1-den}
\alpha^2 \left(\int_\Omega u_\epsilon^6 \,dy\right)^{-1/3} & = (S/3)^{-1/2} + (S/3)^{-2} \frac{8\pi}{3} \phi_0(x)\lambda^{-1}
 - 45  S^{-2} \int_\Omega U_{x,\lambda}^4 w^2\,dy + o(\lambda^{-1}) \,.
\end{align}

\emph{Expansion of the quotient.}
Multiplying \eqref{1-num} and \eqref{1-den} gives 
\begin{align*}
\mathcal S_{a+\epsilon V}[u_\eps] & = S + \lambda^{-1} (S/3)^{-1/2} 4\pi  \phi_0(x) + \lambda^{-1} (S/3)^{-1/2} \int_\Omega a G_0(x,y)^2\,dy \\
& \quad + (S/3)^{-1/2} \int_\Omega \left( |\nabla w|^2 + aw^2 + 2\lambda^{-1/2} a G_0(x,y)w - 15\, U_{x,\lambda}^4 w^2 \right) dy  + o(\lambda^{-1}) \,.
\end{align*}
The resolvent identity together with the symmetry $G_0(x,y)=G_0(y,x)$ implies 
\begin{align*}
& \int_\Omega a(y) G_0(x,y)^2\,dy - (4\pi)^{-1} \iint_{\Omega\times\Omega}  G_0(a,y) a(y) G_a(y,y') a(y')G_0(y',x)\,dy\,dy' \\
&\qquad = \int_\Omega G_0(x,y) a(y) G_a(y,x)\,dy = 4\pi \left( \phi_a(x) - \phi_0(x) \right).
\end{align*}
This completes the proof of the lemma.
\end{proof}

%%%%%%%%%%%%%

\subsection{Regularization and coercivity}

In this subsection we will show that the coercivity bound from Lemma \ref{coercivity} remains essentially true after regularization. A convenient regularization procedure for us is a spectral cut-off. Namely, we denote by $\mathds 1(-\Delta + a \leq \mu\, )$ the spectral projection for the interval $(-\infty,\mu]$ of the self-adjoint operator $-\Delta +a$ in $L^2(\Omega)$ with Dirichlet boundary condition. The parameter $\mu$ here will be later chosen large depending on $\epsilon$.

\begin{lem} \label{lem-regular}
Let $v\in H^1_0(\Omega)$. Then for any $\mu\geq 1$,
\begin{equation}
\label{w-infty}
\| \mathds 1(-\Delta + a \leq \mu\, ) v \|_\infty\  \lesssim  \ \mu^{1/4}\,  \|\nabla v \| \, . 
\end{equation}
\end{lem}

\begin{proof}
Let $a_- = \max\{0,-a\}$. By the maximum principle or the Trotter product formula, we have
\begin{equation} \label{hk-upperb}
0\leq e^{-t(-\Delta+a)}(x,x) \, \leq \, (4\pi t)^{-3/2}\, e^{t\|a_-\|_\infty}
\qquad \text{for all}\ t>0 \,;
\end{equation}
see, e.g., \cite[Thm.~2.4.4]{da} for related estimates.

We denote by $E_n$ the eigenvalues of $-\Delta+a$ in $L^2(\Omega)$ and by $\Phi_n$ the corresponding $L^2$-normalized eigenfunctions. We bound for any $x\in\Omega$
\begin{align*}
\left| \left(\mathds 1(-\Delta + a \leq \mu) v\right)(x) \right|
& = \left| \sum_{E_n\leq\mu} (\Phi_n,v) \Phi_n(x) \right| \\
& \leq \Big( \sum_{E_n\leq\mu} E_n |(\Phi_n,v)|^2 \Big)^{1/2} \Big( \sum_{E_n\leq\mu} E_n^{-1} |\Phi_n(x)|^2 \Big)^{1/2}.
\end{align*}
We clearly have
$$
\sum_{E_n\leq\mu} E_n |(\Phi_n,v)|^2 \leq \sum_n E_n |(\Phi_n,v)|^2 = (v,(-\Delta+a)v) \lesssim \|\nabla v\|^2 \,.
$$
The heat kernel bound \eqref{hk-upperb} implies that for any $s>0$ and $t>0$
$$
\sum_{E_n\leq s} |\Phi_n(x)|^2 \leq e^{t s} \sum_{E_n\leq s} e^{-t E_n} |\Phi_n(x)|^2\,  \leq \, e^{t (s+\|a_-\|_\infty)} \, (4\pi t)^{-3/2} \,,
$$
and choosing $t=(3/2) (s+\|a_-\|_\infty)^{-1}$ we obtain for any $s>0$,
$$
\sum_{E_n\leq s} |\Phi_n(x)|^2 \ \leq \ \left( \frac{e}{6\pi} \right)^{3/2} (s+\|a_-\|_\infty)^{3/2} \,.
$$
Thus, writing $E^{-1} = \int_E^\infty s^{-2} \,ds$, we get 
\begin{align*}
\sum_{E_n\leq\mu} E_n^{-1} |\Phi_n(x)|^2 & = \int_0^\infty \sum_{E_n\leq\mu} |\Phi_n(x)|^2 \mathds 1(E_n\leq s) \,\frac{d s}{s^2}
= \int_{E_1}^\infty \sum_{E_n\leq\min\{\mu, s\}} |\Phi_n(x)|^2 \, \frac{d s}{s^2} \\
& \leq \left( \frac{e}{6\pi} \right)^{3/2} \int_{E_1}^\infty \min\big\{(\mu+\|a_-\|_\infty)^{3/2}, (s+\|a_-\|_\infty)^{3/2}\big\} \, \frac{ds}{s^2} \,.
\end{align*}
The integral is easily seen to be bounded by a universal constant times
$$
\mu^{1/2} + E_1^{-1} \|a_-\|_\infty^{3/2} \,.
$$
This proves the claimed bound.
\end{proof}

\begin{lem} \label{lem-aux}
There are constants $T_*<\infty$, $\rho>0$ and $C<\infty$ such that for all $x\in\Omega$, $\lambda>0$ with $ d\lambda\geq T_*$, and all $v\in T_{x, \lambda}^\bot$ and all $\mu\geq 1$ the function
$$
v_> := \mathds 1(-\Delta + a > \mu) v 
$$
satisfies
\begin{equation}\label{eps-ineq}
\int_\Omega \left( |\nabla v_>|^2 + a v_>^2 -15\, U_{x,\lambda}^4 \, v_>^2 \right) dy\, \geq \, \rho \int_\Omega  |\nabla v_>|^2 \, dy - C \mu^{1/2} \lambda^{-1} \|\nabla v\|^2 \,.
\end{equation}
\end{lem}

\begin{proof}
\emph{Step 1.} We construct an orthonormal basis in $T_{x,\lambda} = {\rm Span} \{\phi_1,\dots, \phi_5\}$, where
$$
\phi_1 = PU_{x,\lambda}, \quad \phi_2 = \partial_\lambda PU_{x,\lambda}, \quad \phi_{j} = \partial_{x_{j-2}} PU_{x,\lambda}, \ \ \ j=3,4,5 \,.
$$
From \cite[Appendix B]{rey2} we know that, as $\lambda\to\infty$, 
\begin{equation} \label{nabla-norm}
\|\nabla \phi_1\| \ \sim \ 1, \quad \|\nabla \phi_2\| \ \sim\ \lambda^{-1}, \quad \|\nabla \phi_j\| \ \sim \ \lambda, \quad j=3,4,5 \,,
\end{equation}
uniformly in $x$ with $\lambda d\geq T_*$, where $T_*$ is any fixed constant. Here $\sim$ means that the quotient of both quantities is bounded from above and away from zero. Let 
\begin{equation} \label{psi-j} 
\tilde\phi_j := \frac{\phi_j}{\|\nabla \phi_j\|} \,, \qquad j=1,\dots ,5 \,,
\end{equation} 
and
$$
G_{j,k}:=\int_\Omega \nabla\tilde\phi_j \cdot \nabla\tilde\phi_k \,dy \,, \qquad j,k=1,\ldots, 5 \,.
$$
By \cite[Appendix B]{rey2} and \eqref{nabla-norm},
\begin{equation}
\label{eq:overlap}
G_{j,k}:= \mathcal{O}(\lambda^{-1}) \quad \text{for all}\ j\neq k \ 
\qquad\text{and}\qquad
G_{j,j} = 1 \quad \text{for all}\ j \,.
\end{equation}
Hence, if $\lambda$ is large enough, which follows from $d \lambda\geq T_*$ with sufficiently large $T_*$ since $\Omega$ is bounded, then $G$ is invertible and
\begin{equation}
\label{eq:gramschmidt}
(G^{-1/2})_{j,k} = \delta_{j,k} + \mathcal O(\lambda^{-1}) \,.
\end{equation}
Hence, by the Gram--Schmidt procedure,
\begin{equation}\label{eq:gramschmidt2}
\psi_j := \sum_k (G^{-1/2})_{j,k}\,  \tilde\phi_k
\qquad j = 1,\ldots,5 \,,
\end{equation}
is an $H^1_0(\Omega)$-orthonormal basis of $T_{x,\lambda}$.

\medskip

\emph{Step 2.} We decompose
\begin{equation}
\label{eq:decomposition u}
v_> = v_\parallel + v_\bot
\qquad\text{with}\ \  v_\parallel\in T_{x,\lambda}\ \ \text{and}\  \ \ 
 v_\bot \in T_{x,\lambda}^\bot
\end{equation}
and claim that
\begin{equation}
\label{eq:betacoeff0}
\|\nabla v_\parallel \| = \mathcal O( \lambda^{-1/2} \mu^{1/4}\,  \|\nabla v\|) \,.
\end{equation}

Since the $\psi_j$ are an orthonormal basis of $T_{x,\lambda}$, we have
$$
v_\parallel = \sum_{j=1}^5 m_j \psi_j
\qquad\text{with}\qquad
m_j := \int_\Omega \nabla\psi_j\cdot\nabla v_>\,dy \,.
$$
Since
$$
\int_\Omega |\nabla v_\parallel|^2\,dy = \sum_j m_j^2 \,,
$$
the claim \eqref{eq:betacoeff0} follows from
\begin{equation}
\label{eq:betacoeff}
m_j = \mathcal O( \lambda^{-1/2} \mu^{1/4}\,  \|\nabla v\|)
\qquad\text{for all}\ j=1,\ldots,5 \,.
\end{equation}

In order to prove the latter, we introduce
$$
\ell_j := \int_\Omega \nabla \tilde\phi_j \cdot \nabla v_> \ dy \, ,
$$
so that, by \eqref{eq:gramschmidt2},
$$
m_j = \sum_k (G^{-1/2})_{j,k}\, l_k \,.
$$
Therefore, in view of \eqref{eq:gramschmidt}, the claim \eqref{eq:betacoeff} follows from
\begin{equation}
\label{eq:alphacoeff}
\ell_j = \mathcal O( \lambda^{-1/2} \mu^{1/4} \|\nabla v\|)
\qquad\text{for all}\ j=1,\ldots,5 \,.
\end{equation}
To prove \eqref{eq:alphacoeff}, we use the fact that $v \in T_{x,\lambda}^\perp$ to find
$$
\ell_j = - \int_\Omega \nabla\tilde\phi_j\cdot\nabla v_< \,dy = \int_\Omega v_<\,  \Delta\tilde\phi_j \,dy \,.
$$
Thus,
$$
|\ell_j| \leq \|v_<\|_\infty\,  \|\Delta\tilde\phi_j\|_1 \,.
$$
According to \eqref{w-infty} we have $ \|v_<\|_\infty \lesssim \mu^{1/4} \|\nabla v\|$. Thus, in order to complete the proof of \eqref{eq:alphacoeff} we need to show that $\|\Delta\tilde\phi_j\|_1=\mathcal O(\lambda^{-1/2})$ for $j=1,\ldots,5$. We have 
\begin{align}\label{eq:laplzeromodes1}
-\Delta\tilde\phi_1 & = \|\nabla\phi_1\|^{-1} 3\, U_{x,\lambda}^5 \,,\qquad\qquad -\Delta\tilde\phi_2 = \|\nabla\phi_2\|^{-1} 15\, U_{x,\lambda}^4\partial_\lambda U_{x,\lambda} \,, \notag \\
-\Delta\tilde\phi_j & = \|\nabla\phi_j\|^{-1} 15\, U_{x,\lambda}^4\partial_j U_{x,\lambda} \qquad\text{for}\ j=3,4,5 \,.
\end{align}
Thus, the claimed bound on $\|\Delta\tilde\phi_j\|_1$ follows from $\eqref{nabla-norm}$ and straightforward bounds on $\|U_{x,\lambda}\|_5$, $\|\partial_\lambda U_{x,\lambda}\|_5$ and $\|\partial_j U_{x,\lambda}\|_5$. This completes the proof of \eqref{eq:alphacoeff} and therefore of \eqref{eq:betacoeff0}.

\medskip

\emph{Step 3.} By the orthogonal decomposition \eqref{eq:decomposition u} we have
$$
\int_\Omega |\nabla v_>|^2\,dy = \int_\Omega |\nabla v_\parallel|^2 \,dy + \int_\Omega |\nabla v_\bot|^2 \,dy \,.
$$
Moreover, we bound, with a parameter $\delta>0$ to be determined, 
$$
\int_\Omega U_{x,\lambda}^4\, v_>^2 \,dy \ \leq \  (1+\delta^{-1}) \int_\Omega U_{x,\lambda}^4\, v_\parallel^2\,dy + (1+\delta) \int_\Omega U_{x,\lambda}^4\, v_\bot^2\,dy
$$
and
$$
\int_\Omega a\, v_>^2 \,dy \ \geq \  -(1+\delta^{-1}) \int_\Omega |a|\, v_\parallel^2\,dy + \int_\Omega a\, v_\bot^2\,dy - \delta \int_\Omega |a|\, v_\bot^2\,dy \,.
$$
Thus,
\begin{align*}
\int_\Omega \left(|\nabla v_>|^2 + a v_>^2 - 15\, U_{x,\lambda} v_>^2 \right)dy
& \geq \int_\Omega \left(|\nabla v_\bot|^2 + a v_\bot^2 - 15\, U_{x,\lambda} v_\bot^2 \right)dy
- \delta \int_\Omega (|a|+15\, U_{x,\lambda}^4) v_\bot^2\,dy \\
& \quad + \int_\Omega |\nabla v_\parallel|^2\,dy - (1+\delta^{-1}) \int_\Omega (|a|+15\, U_{x,\lambda}^4) v_\parallel^2\,dy \,.
\end{align*}
Clearly,
\begin{equation} \label{hs-ineq}
 \int_\Omega (|a|+15\,U_{x,\lambda}^4)\, z^2 \,dy \leq \left( \|a\|_{3/2} + 15 \|U_{x,\lambda}\|_6^4 \right) \|z\|_6^2 \lesssim\ \|\nabla z\|^2 \qquad \forall\, z\in H_0^1(\Omega).
\end{equation}
Since $v_\bot\in T_{x,\lambda}^\bot$, Lemma \ref{coercivity} and \eqref{hs-ineq} imply that, after increasing $T_*$ if necessary, there are $\delta>0$ and $c>0$ such that
$$
\int_\Omega \left(|\nabla v_\bot|^2 + a v_\bot^2 - 15\, U_{x,\lambda} v_\bot^2 \right)dy
- \delta \int_\Omega (|a|+15\, U_{x,\lambda}^4) v_\bot^2\,dy \geq c \int_\Omega |\nabla  v_\bot|^2\,dy \,.
$$
On the other hand, by \eqref{hs-ineq} and \eqref{eq:betacoeff0},
\begin{align*}
\int_\Omega (|a|+15\, U_{x,\lambda}^4) v_\parallel^2\,dy & \lesssim \int_\Omega |\nabla v_\parallel|^2\,dy = \mathcal O( \lambda^{-1} \mu^{1/2} \|\nabla v\|^2) \,.
\end{align*}
This completes the proof of Lemma \ref{lem-aux}.
\end{proof}

%%%%%%%%%%%%%%%

\subsection{Completing the square}

The following lemma gives a lower bound on the term in \eqref{2-S} which involves $w$. As explained above, this is the crucial step in the proof of Proposition \ref{thm-q}.  

\begin{lem} \label{prop-1}
For some constant $c>0$,
\begin{align}
\label{w-lowerb} 
& \int_\Omega \left( |\nabla w|^2 + aw^2 + 2 \lambda^{-1/2} a G_0(x,y)w - 15 \, U_{x,\lambda}^4 w^2 \right) dy \nonumber \\
& \qquad \geq -\lambda^{-1} (4\pi)^{-1} \iint_{\Omega\times\Omega}  G_0(x,y) a(y) G_a(y,y') a(y')G_0(y',x)\,dy\,dy' \nonumber \\
& \qquad \quad + c\,  \Big\| (-\Delta+a)^{1/2} w + (-\Delta+a)^{-1/2} \lambda^{-1/2} a G_0(x,\cdot) \Big\|^2 + \mathcal O(\lambda^{-3/2}) \,. 
\end{align}
\end{lem}

\begin{proof}
For a parameter $\mu\geq 1$ to be specified later we decompose $w=w_>+w_<$ with
$$
w_> = \mathds 1(-\Delta + a > \mu) w \,,
\qquad
w_< = \mathds 1(-\Delta + a \leq \mu) w \,.
$$
Then 
\begin{equation} \label{w-ort}
\int_\Omega \left( |\nabla w|^2 + aw^2\right)dy = \int_\Omega \left( |\nabla w_>|^2 + aw_>^2\right)dy + \int_\Omega \left( |\nabla w_<|^2 + aw_<^2\right)dy
\end{equation}
and therefore, for any $\delta>0$,
\begin{align} \label{id-w}
& \int_\Omega \left( |\nabla w|^2 + aw^2 + 2\lambda^{-1/2} a G_0(x,y)w - 15 \,U_{x,\lambda}^4 w^2 \right) dy\,  \geq\, I_< + I_> + R_<(\delta) + R_>(\delta) \,,
\end{align}
where
\begin{align*}
I_< & := \int_\Omega \left( |\nabla w_<|^2 + aw_<^2 + 2 \lambda^{-1/2} a G_0(x,y) w_< \right) dy \,, \\
I_> & := \int_\Omega \left( |\nabla w_>|^2 + aw_>^2 - 15\, U_{x,\lambda}^4w_>^2 \right) dy \,, \\
R_<(\delta) & := - 15\,(1+\delta^{-1}) \int_\Omega U_{x,\lambda}^4 w_<^2\,dy \,, \\
R_>(\delta) & := -15\,\delta \int_\Omega U_{x,\lambda}^4 w_>^2\,dy +2\lambda^{-1/2} \int_\Omega a G_0(x,y)w_>\,dy \,.
\end{align*}
By completing the square we find
\begin{align*}
I_< & = -\lambda^{-1} (4\pi)^{-1} \iint_{\Omega\times\Omega}  G_0(x,y) a(y) G_a(y,y') a(y')G_0(y',x)\,dy\,dy' \\
& \quad  + \left\| (-\Delta+a)^{1/2} w_< + (-\Delta+a)^{-1/2} \lambda^{-1/2} a G_0(x,\cdot) \right\|^2 \,,
\end{align*}
and with $0\leq c\leq 1$ to be determined we estimate
\begin{align}
I_< & \geq -\lambda^{-1} (4\pi)^{-1} \iint_{\Omega\times\Omega}  G_0(x,y) a(y) G_a(y,y') a(y')G_0(y',x)\,dy\,dy' 
\nonumber \\ & \quad  \
+ c \left\| (-\Delta+a)^{1/2} w_< + (-\Delta+a)^{-1/2} \lambda^{-1/2} a G_0(x,\cdot) \right\|^2 \nonumber \\
& = -\lambda^{-1} (4\pi)^{-1} \iint_{\Omega\times\Omega}  G_0(x,y) a(y) G_a(y,y') a(y')G_0(y',x)\,dy\,dy'  \nonumber \\
& \quad + c \left\| (-\Delta+a)^{1/2} w + (-\Delta+a)^{-1/2} \lambda^{-1/2} a G_0(x,\cdot) \right\|^2
 \nonumber \\& \quad
 - c \left\| (-\Delta+a)^{1/2} w_> \right\|^2 - 2c \lambda^{-1/2} \int_\Omega a G_0(x,y)w_> \,dy \,. \label{i<}
\end{align}
According to Lemma \ref{lem-aux} there are $\rho>0$ and $C<\infty$ such that for all sufficiently small $\epsilon>0$,
$$
I_> \geq \rho \int_\Omega |\nabla w_>|^2\,dy - C \mu^{1/2} \lambda^{-1} \|\nabla w\|^2 \,.
$$
Since $a \in L^\infty(\Omega)$, we have
\begin{equation}
\label{operator norm upper bound }
\| (-\Delta + a)^{1/2} z\|^2 \leq C' \, \|\nabla z\|^2 \qquad \forall\,  z \in H^1_0(\Omega) \,.
\end{equation}
We apply this with $u=w_>$ and infer that
\begin{align*}
I_< + I_>(\delta) + R_<(\delta) + R_>
& \geq -\lambda^{-1} (4\pi)^{-1} \iint_{\Omega\times\Omega}  G_0(x,y) a(y) G_a(y,y') a(y')G_0(y',x)\,dy\,dy'  \nonumber \\
& \quad + c \left\| (-\Delta+a)^{1/2} w + (-\Delta+a)^{-1/2} \frac{\alpha}{\sqrt\lambda} a G_0(x,\cdot) \right\|^2\nonumber + R_1(\delta) + R_2(\delta)
\end{align*}
where
\begin{align*}
R_1(\delta) & = \rho \|\nabla w_>\|^2 - c C' \|\nabla w_>\|^2 - 15\, \delta \int_\Omega U_{x,\lambda}^4 w_>^2\,dy \,, \\
R_2(\delta) & = - C \mu^{1/2} \lambda^{-1} \|\nabla w\|^2 + 2(1-c)\lambda^{-1/2} \int_\Omega a G_0(x,y)w_>\,dy
- 15\,(1+\delta^{-1}) \int_\Omega U_{x,\lambda}^4 w_<^2\,dy \,.
\end{align*}
We now choose $c=\min\{1,\rho/(2C')\}$. Moreover, by \eqref{hs-ineq} we can choose a $\delta>0$, independent of $\epsilon$ and $\mu$ such that
$$
R_1(\delta)\geq 0 \,.
$$
From now on, we fix this value of $\delta$. 

It remains to show that $R_2(\delta)$ is $\mathcal O(\lambda^{-3/2})$ for an appropriate choice of $\mu$. By \eqref{esp-w} and \eqref{operator norm upper bound } and by the orthogonality \eqref{w-ort} we have 
\begin{equation} \label{w-ort-2}
\mathcal O(\lambda^{-1}) = \int_\Omega |\nabla w|^2\,dy \gtrsim \int_\Omega \left(|\nabla w|^2 + a w^2\right) dy \geq  \int_\Omega \left(|\nabla w_>|^2 + a w_>^2\right) dy \geq\,  \mu\,  \|w_>\|^2 \,.
\end{equation}
Thus, since $a\in L^\infty(\Omega)$ and since $G_0(x, \cdot)$ is uniformly bounded in $L^2(\Omega)$, we have
$$
\left| \int_\Omega a\,  G_0(x,y)w_>\,dy \right| \lesssim \|w_>\| \lesssim \mu^{-1/2} \lambda^{-1/2} \,.
$$
Moreover, by Lemma \ref{lem-regular},
$$
\int_\Omega U_{x,\lambda}^4 w_<^2\,dy \leq \|w_<\|_\infty^2 \int_\Omega U_{x,\lambda}^4\,dy \lesssim  \mu^{1/2} \|\nabla w\|^2 \int_{\R^3} U_{x,\lambda}^4\,dy \lesssim \mu^{1/2}\lambda^{-2} \,.
$$
Thus,
$$
R_2(\delta) \gtrsim - \left( \mu^{1/2} \lambda^{-2} + \mu^{-1/2} \lambda^{-1} \right).
$$
With the choice $\mu=\lambda$ the right side becomes $\mathcal O(\lambda^{-3/2})$, as claimed.
\end{proof}

Now we prove the main result of this section. 

\begin{proof}[Proof of Proposition \ref{thm-q}]
Inserting \eqref{w-lowerb} into \eqref{2-S} gives 
\begin{align}\label{S eps lower bound}
\mathcal S_{a+\epsilon V}[u_\eps] & \geq S + 4\pi \, \lambda^{-1} (S/3)^{-1/2} \phi_a(x) \nonumber \\
& \quad + (S/3)^{-1/2} c \left\| (-\Delta+a)^{1/2} w + (-\Delta+a)^{-1/2} \lambda^{-1/2} a G_0(x,\cdot) \right\|^2 + o(\lambda^{-1}) \,. 
\end{align}
We subtract $S(a+\epsilon V)$ from both sides, multiply by $\lambda$ and take the limsup as $\epsilon\to 0+$. Using the second relation in \eqref{eq:energybound} we obtain
\begin{align*}
0 \geq \limsup_{\epsilon\to 0} &  \left( \lambda (S-S(a+\epsilon V)) + 4\pi (S/3)^{-1/2} \phi_a(x) \right. \\
& \ \left. + (S/3)^{-1/2} c \lambda \left\| (-\Delta+a)^{1/2} w + (-\Delta+a)^{-1/2} \lambda^{-1/2} a G_0(x,\cdot) \right\|^2 \right).
\end{align*}
Since the three terms in the limsup are all non-negative (which for $\phi_a$ follows from Corollary \ref{cor-2}), we deduce that
$$
\lambda (S-S(a+\epsilon V)) = o(1) \,,
\qquad
\phi_a(x) = o(1)
$$
and
$$
\left\| (-\Delta+a)^{1/2} w + (-\Delta+a)^{-1/2} \lambda^{-1/2} a G_0(x,\cdot) \right\|^2 = o(\lambda^{-1}) \,.
$$
Since $-\Delta+a$ is coercive, the last bound implies
$$
\Big\| \nabla \Big( w + (-\Delta+a)^{-1} \lambda^{-1/2} a G_0(x,\cdot) \Big) \Big\|^2 = o(\lambda^{-1}) \,.
$$
By the resolvent identity,
$$
(-\Delta+a)^{-1} a G_0(x,\cdot) = G_0(x,\cdot) - G_a(x,\cdot) = H_a(x,\cdot) - H_0(x,\cdot) \,,
$$
and therefore, setting $q:=w+\lambda^{-1/2} (H_a(x,\cdot)-H_0(x,\cdot))$, the previous bound can be rewritten as $\|\nabla q\|^2= o(\lambda^{-1})$. This completes the proof of the proposition.
\end{proof}

%%%%%%%%%%%%%%%%%%%%%%%%%%%%%%%%%%%%%%%%%%%%%%%%%%%%%%%%%
\section{\bf A refined decomposition of almost minimizers}
\label{sec-new-decomp}

From Proposition \ref{thm-q} we infer that any sequence $(u_\eps)$ satisfying \eqref{appr-min} can be decomposed as 
$$
u_\epsilon=  \alpha \left( \psi_{x, \lambda} + q \right),
$$
where
$$
\psi_{x,\lambda} = PU_{x, \lambda} - \lambda^{-1/2} (H_a(x, \cdot)- H_0(x, \cdot))
$$
is as in the proof of the upper bound, see \eqref{eq:psi}, and where 
$$
\|\nabla q\| = o(\lambda^{-1/2}) \,.
$$
Thus, expanding $\mathcal S_{a+\epsilon V}[u_\epsilon]$ leads to an expression that coincides with the upper bound in Corollary \ref{cor-2} up to additional terms involving $q$. Using coercivity we will be able to show that the contribution from
$$
r := \Pi_{x,\lambda}^\bot q \,,
$$
the orthogonal projection of $q$ onto $T_{x,\lambda}^\bot$ in $H^1_0(\Omega)$, is negligible; see Lemma \ref{lem-g} below. The main focus in this section is on
$$
\Pi_{x,\lambda} q = \Pi_{x,\lambda} \left( w + \lambda^{-1/2} (H_a(x,\cdot)-H_0(x,\cdot)) \right) = \lambda^{-1/2}\, \Pi_{x,\lambda} (H_a(x,\cdot)-H_0(x,\cdot)),
$$
where the last identity follows from $w\in T_{x,\lambda}^\bot$. In Lemma \ref{expfinalq} we will prove that the contribution from $\Pi_{x,\lambda} q$ is negligible. This is not obvious and, in fact, somewhat surprising since $\Pi_{x,\lambda}q$ is of order $\lambda^{-1}$ and not smaller.

%%%%%%%%%%%%%%%%%%%%%%%%%%%%%%%%%%%%%%

\subsection{Preliminary estimates}

Let us write
$$
\Pi_{x,\lambda} q = \beta \lambda^{-1} PU_{x,\lambda} + \gamma \partial_\lambda PU_{x,\lambda} + \sum_{j = 1}^3 \delta_j\, \lambda^{-3}  \partial_{x_j} PU_{x,\lambda} \,.
$$
Since $PU_{x,\lambda}$, $\partial_\lambda PU_{x,\lambda}$ and $\partial_{x_j} PU_{x,\lambda}$, $j=1,2,3$, are linearly independent for sufficiently large $\lambda$, the numbers $\beta$, $\gamma$ and $\delta_j$, $j=1,2,3$, (depending on $\epsilon$, of course) are uniquely determined. The choice of the different powers of $\lambda$ multiplying these coefficients is motivated by the following lemma.

\begin{lem}
\label{lem-bg}
As $\epsilon\to 0$, we have 
$$
\beta, \;  \gamma, \; \delta_j = \mathcal{O}(1) .
$$
\end{lem}

\begin{proof}
We recall that the functions $\tilde\phi_j$, $j=1,\ldots,5$, were introduced in \eqref{psi-j}. Let
$$
a_j := \int_\Omega \nabla\tilde\phi_j\cdot\nabla q\,dy \,,
\qquad j =1,\ldots,5 \,.
$$

\emph{Step 1.} We shall show that
\begin{equation}
\label{eq:orthocoeff}
a_1, a_2 = \mathcal O(\lambda^{-1}) \,,
\qquad
a_3, a_4, a_5 =\mathcal O(\lambda^{-2}) \,.
\end{equation}

Since $-\lambda^{-1/2}(H_a(x,\cdot) - H_0(x,\cdot)) + q = w \in T_{x,\lambda}^\perp$, we have
\begin{align*}
a_j & = \lambda^{-1/2} \int_\Omega \nabla\tilde\phi_j\cdot\nabla_y (H_a(x,y)-H_0(x,y))\,dy 
= - \lambda^{-1/2} \int_\Omega (\Delta\tilde\phi_j) (H_a(x,y)-H_0(x,y))\,dy \,.
\end{align*}
Formulas for the Laplacians $\Delta\tilde\phi_j$ are given in \eqref{eq:laplzeromodes1} and the quantities $\|\nabla\phi_j\|$ appearing there were estimated in \eqref{nabla-norm}. For $a_1$, the integral $\int_\Omega U_{x,\lambda}^5 (H_a(x,y)-H_0(x,y))\,dy$ is $\mathcal O(\lambda^{-1/2})$ according to Lemma \ref{lem-uh}, which proves the claim in \eqref{eq:orthocoeff}. To bound $a_j$ for $j=2,\ldots,5$ we compute
\begin{align*}
\partial_\lambda U_{x, \lambda}(y) = \frac{\lambda^{-1/2}}{2} \frac{1 - \lambda^2|y-x|^2}{(1+\lambda^2|y-x|^2)^{3/2}} \,, \qquad
\partial_{x_i} U_{x, \lambda}(y) = \lambda^{5/2} \frac{y_i-x_i}{(1+\lambda^2|y-x|^2)^{3/2}} \,, \quad i = 1,2,3.
\end{align*}
This expression and straightforward bounds lead to the claim for $a_2$ in \eqref{eq:orthocoeff}.

To prove \eqref{eq:orthocoeff} for $a_j$ with $j=3,4,5$ we need to bound
$$
\int_\Omega  (H_a(x,y) - H_0(x,y)) U_{x,\lambda}^4 \partial_{x_j} U_{x,\lambda}\,dy \,.
$$
From Step 1 in the proof of Lemma \ref{lem-uh}, recalling \eqref{eq-d}, we infer that there are $\rho>0$ and $C>0$, both independent of $\epsilon$, such that
$$
\left| H_a(x,y) - H_0(x,y) - H_a(x,x) + H_0(x,x) \right| \lesssim |y-x| 
\qquad\text{for all}\ y\in B_\rho(x) \,.
$$
Since the function $U_{x,\lambda}^4 \partial_{x_j} U_{x,\lambda}$ is odd, we have
$$
\int_{B_\rho(x)}  (H_a(x,x) - H_0(x,x)) U_{x,\lambda}^4 \partial_{x_j} U_{x,\lambda}\,dy = 0 \,.
$$
On the other hand, using the above expression for $\partial_{x_j} U_{x,\lambda}$ we find
$$
\int_{\Omega} \min\{|y-x|,\rho\} \left| U_{x,\lambda}^4 \partial_{x_j} U_{x,\lambda} \right| dy = \mathcal O(\lambda^{3/2}) \,.
$$
This proves \eqref{eq:orthocoeff} for $j=3,4,5$.

\emph{Step 2.} Let us deduce the statement of the lemma. We have
$$
\Pi_{x,\lambda} q = \sum_{j=1}^5 \tilde a_j \tilde\phi_j
$$
with
$$
\tilde a_1 := \beta \lambda^{-1} \|\nabla PU_{x,\lambda} \| \,,
\qquad
\tilde a_2 := \gamma \|\nabla\partial_\lambda PU_{x,\lambda} \| \,,
\qquad
\tilde a_j := \delta_j \lambda^{-3} \|\nabla \partial_{x_{j-2}} PU_{x,\lambda} \| \,,\  j=3,4,5 \,.
$$
In view of \eqref{nabla-norm}, the assertion of the lemma is equivalent to
\begin{equation}
\label{eq:orthocoeff2}
\tilde a_1, \tilde a_2 =\mathcal O( \lambda^{-1}) \,,
\qquad
\tilde a_j = \mathcal O( \lambda^{-2}) \,,\ j=3,4,5 \,.
\end{equation}
With respect to the orthonormal system $\psi_j$, $j=1,\ldots,5$, from \eqref{eq:gramschmidt2} we have
$$
\Pi_{x,\lambda} q = \sum_{j=1}^5 (\nabla\psi_j,\nabla q) \psi_j \,. 
$$
Using \eqref{eq:gramschmidt2} twice to express $\psi_j$ in terms of $\tilde\phi_k$'s we obtain
$$
\Pi_{x,\lambda} q = \sum_{k=1}^5 \sum_{\ell=1}^5 (G^{-1})_{k,\ell} (\nabla\tilde\phi_\ell,\nabla q) \, \tilde\phi_k = \sum_{k=1}^5 \sum_{\ell=1}^5 (G^{-1})_{k,\ell}\, a_\ell \, \tilde\phi_k \,.
$$
Thus,
$$
\tilde a_k = \sum_{\ell=1}^5 (G^{-1})_{k,\ell}\, a_\ell \,,
\qquad k=1,\ldots, 5 \,.
$$
Similarly as in \eqref{eq:gramschmidt} one finds
\begin{equation*}
(G^{-1})_{j,k} = \delta_{j,k} + \mathcal O(\lambda^{-1}) \,,
\end{equation*}
and then \eqref{eq:orthocoeff2} follows from \eqref{eq:orthocoeff}. This completes the proof of the lemma.
\iffalse
Writing $G=1-(1-G)$ and expanding into a Neumann series we obtain, using \eqref{eq:overlap}, that
$$
(G^{-1})_{k,\ell} = (2- G)_{k,\ell} + \mathcal O(\lambda^{-2}) = 2\delta_{k,\ell} - \int_\Omega \nabla\tilde\phi_k\cdot\nabla\tilde\phi_\ell\,dy + \mathcal O(\lambda^{-2}) \,.
$$
When $k=1,2$ it suffices to use \eqref{eq:overlap} again to obtain $(G^{-1})_{k,\ell} = \delta_{k,\ell} + \mathcal O(\lambda^{-1})$ to deduce \eqref{eq:orthocoeff2} from \eqref{eq:orthocoeff}. When $k=3,4,5$, we use the fact \cite[Appendix B]{rey2} that
\[ \int_\Omega \nabla \tilde\phi_k \cdot \nabla \tilde\phi_\ell = \mathcal O(\lambda^{-2}) \quad \text{ for } \quad k = 3,4,5, \quad \ell = 1,...,5, \quad  k \neq \ell \,. \]
Thus, in this case $(G^{-1})_{k,\ell} = \delta_{k,\ell} + \mathcal O(\lambda^{-2})$ and \eqref{eq:orthocoeff2} follows again from \eqref{eq:orthocoeff}. This completes the proof of the lemma.
\fi
\end{proof}

\begin{rem}\label{coeffsize}
The same method of proof shows that there are non-zero numbers $\beta_0,\gamma_0,\delta_{0,j}$ such that
$$
\beta\to\beta_0 \,,
\quad
\gamma\to\gamma_0 \,,
\quad
\delta_{0,j}\to\delta_0
$$
as $\epsilon\to 0$. Indeed, proceeding as in Step 1 above one can show that $\lambda a_k$ for $k=1,2$ and $\lambda^2 a_k$ for $k=3,4,5$ have a non-zero limit as $\epsilon\to 0$. As in Step 2 above, this implies that $\lambda \tilde a_k$ for $k=1,2$ have a non-zero limit as $\epsilon\to 0$. In order to compute the limits of $\lambda \tilde a_k$ for $k=3,4,5$ one needs to use, in addition, the fact that $(G^{-1})_{k,\ell} = \delta_{k,\ell} + \mathcal O(\lambda^{-2})$ for $k=3,4,5$. Indeed, by a Neumann series for $G=1-(1-G)$ one finds
$$
(G^{-1})_{k,\ell} = (2- G)_{k,\ell} + \mathcal O(\lambda^{-2}) = 2\delta_{k,\ell} - \int_\Omega \nabla\tilde\phi_k\cdot\nabla\tilde\phi_\ell\,dy + \mathcal O(\lambda^{-2}) \,,
$$
and then one can use bounds from \cite[Appendix B]{rey2} for the integral on the right side.
\end{rem}

%%%%%%%%%%%%%%%%%%%

\subsection{A third expansion}
\label{section expansion}

In this subsection, we shall prove the following lemma.

\begin{lem}\label{expfinalq}
As $\epsilon\to 0$,
\begin{align}\label{eq:expfinalq}
\mathcal S_{a+\epsilon V}[u_\epsilon] & = \mathcal S_{a+\epsilon V}[\psi_{x,\lambda}] + (S/3)^{-1/2} \left( \mathcal E_0[r] - \frac{N_0}{3\,D_0}\mathcal I[r] \right)
+ o(\lambda^{-2}) + o(\epsilon \lambda^{-1})
\end{align}
with 
\begin{equation}
\label{eq:defnd}
N_0 := \int_\Omega \left( |\nabla \psi_{x, \lambda}|^2 + (a+\epsilon V)\psi_{x, \lambda}^2\right)dy, 
\qquad 
D_0 := \int_\Omega \psi_{x,\lambda}^6\,dy
\end{equation}
and 
\begin{align}\label{eq:defi}
\mathcal I[r] & :=  -30 \, \lambda^{-1/2} \int_\Omega U_{x,\lambda}^4 H_a(x,y)r\,dy + 15 \int_\Omega U_{x,\lambda}^4 r^2\,dy + 20 \int_\Omega U_{x,\lambda}^3 r^3\,dy \, .
\end{align}
\end{lem}

We emphasize that the coefficients $\beta$, $\gamma$ and $\delta_j$ enter only into the remainders $o(\lambda^{-2}) + o(\epsilon \lambda^{-1})$. This is somewhat surprising since $\beta$ enters to orders $\lambda^{-1}$ and $\lambda^{-2}$ and $\gamma$ enters to order $\lambda^{-2}$ in the expansion of the numerator and the denominator.

In the following, it will be convenient to abbreviate
$$
g := \beta \lambda^{-1} PU_{x,\lambda} + \gamma \partial_\lambda PU_{x,\lambda} \,,
\qquad
h := \sum_{j=1}^3 \delta_j \lambda^{-3} \partial_{x_j} PU_{\lambda,x} \,,
$$
so that
$$
u = \alpha ( \psi_{x,\lambda} + g + h + r) \,.
$$
We record the bounds
\begin{equation}
\label{eq:hbound}
\|\nabla g\| = \mathcal O(\lambda^{-1}) \,,
\qquad
\|\nabla h\| = \mathcal O(\lambda^{-2}) \,,
\qquad
\|\nabla r\| = o(\lambda^{-1/2}) \,.
\end{equation}
Indeed, the bounds on $g$ and $h$ follow from Lemma \ref{lem-bg} together with \eqref{nabla-norm} and that for $r$ follows from Proposition \ref{thm-q} since, by orthogonality, $\|\nabla r\| \leq \|\nabla q\|$.

We will also use the fact that
\begin{equation}
\label{eq:hboundlapl1}
\|\Delta h\|_1 = \mathcal O(\lambda^{-5/2}) \,.
\end{equation}
This follows from Lemma \ref{lem-bg} together with \eqref{eq:laplzeromodes1} and the same bounds that led to \eqref{eq:alphacoeff}.

We will obtain Lemma \ref{expfinalq} from separate expansions of the numerator and the denominator, which we state in the following two lemmas.

%%%%%%%%%%

\subsubsection*{Expanding the numerator}
\label{ssec expanding numerator}

We abbreviate
$$
\mathcal E_\epsilon[v] := \int_\Omega \left( |\nabla v|^2 + (a+\epsilon V)v^2\right)dy
$$
and write $\mathcal E_\epsilon[v_1,v_2]$ for the associated bilinear form. Recall that $N_0$ was defined in \eqref{eq:defnd}. We shall show

\begin{lem}\label{expfinaln}
As $\epsilon\to 0$,
$$
\alpha^{-2} \mathcal E_\epsilon[u_\epsilon] = N_0 + N_1 + \mathcal E_0[r] + o(\lambda^{-2}) + o(\epsilon \lambda^{-1})\, ,
$$
where
$$
N_1 := \int_\Omega |\nabla g|^2\,dy + 2\,\mathcal E_0[\psi_{x,\lambda},g] \,.
$$
\end{lem}

\begin{proof}
\emph{Step 1.} We show that the contribution from $h$ to $\alpha^{-2} \mathcal E_\epsilon[u_\epsilon]$ is negligible, that is,
\begin{equation}
\label{eq:expfinaln1}
\alpha^{-2} \mathcal E_\epsilon[u_\epsilon] = \mathcal E_\epsilon[\psi_{x,\lambda}+g+r] + o(\lambda^{-5/2}) \,.
\end{equation}
Indeed,
$$
\alpha^{-2} \mathcal E_\epsilon[u_\epsilon] = \mathcal E_\epsilon[\psi_{x,\lambda}+g+r] + 2\, \mathcal E_\epsilon[\psi_{x,\lambda}+g+r,h] + \mathcal E_\epsilon[h] \,.
$$
Since $\mathcal E_\epsilon[v_1,v_2]\lesssim \|\nabla v_1\| \|\nabla v_2\|$ for all $v_1,v_2\in H^1_0(\Omega)$, we immediately conclude from \eqref{eq:hbound} that
$$
\mathcal E_\epsilon[h] = \mathcal O(\lambda^{-4})
\qquad
\mathcal E_\epsilon[g+r,h] = o(\lambda^{-5/2}) \,.
$$
Next, using \eqref{eq:hboundlapl1}, \eqref{sup-h-2} and \eqref{sup-h},
\begin{align*}
\int_\Omega \nabla\psi_{x,\lambda}\cdot\nabla h\,dy & = \int_\Omega \nabla PU_{x,\lambda}\cdot\nabla h\,dy + \mathcal O( \lambda^{-1/2} \| H_a(x,\cdot)-H_0(x,\cdot)\|_\infty \|\Delta h\|_1 ) \\
& = \int_\Omega \nabla PU_{x,\lambda}\cdot\nabla h\,dy + \mathcal O(\lambda^{-3}) \,.
\end{align*}
Moreover, by \eqref{eq:overlap} and \eqref{nabla-norm}, 
\begin{align*}
\int_\Omega \nabla PU_{x,\lambda} \cdot\nabla h\,dy
= \sum_{j=1}^3\delta_j \lambda^{-3} \int_\Omega \nabla PU_{x,\lambda}\cdot\nabla \partial_{x_j} PU_{x,\lambda}\,dy 
= \mathcal O(\lambda^{-3}) \,.
\end{align*}
Finally, by \eqref{eq:psibound65} and \eqref{eq:hbound},
$$
\left| \int_\Omega (a+\epsilon V)\psi_{x,\lambda} h\,dy \right| \leq \|a+\epsilon V\|_\infty \|\psi_{x,\lambda}\|_{6/5} \|h\|_6 = \mathcal O(\lambda^{-5/2})
$$
This proves \eqref{eq:expfinaln1}.

\emph{Step 2.} We now extract the relevant contribution from $g$ and show
\begin{equation}
\label{eq:expfinaln2}
\mathcal E_\epsilon[\psi_{x,\lambda}+g+r] = \mathcal E_\epsilon[\psi_{x,\lambda}+r] + 2\, \mathcal E_0[\psi_{x,\lambda},g] + \int_\Omega |\nabla g|^2\,dy + o(\lambda^{-2}) \,.
\end{equation}

Indeed,
$$
\mathcal E_\epsilon[\psi_{x,\lambda}+g+r] = \mathcal E_\epsilon[\psi_{x,\lambda}+r] + 2\,\mathcal E_\epsilon[\psi_{x,\lambda}+r,g] + \mathcal E_\epsilon[g] \,.
$$
By Lemma \ref{lem-bg}, \eqref{eq:zm65}, \eqref{eq:zm65der} and \eqref{eq:hbound},
\begin{align*}
\left| \int_\Omega (a+\epsilon V) (2rg +g^2)\,dy \right| & \leq \|a+\epsilon V\|_\infty \|g\|_{6/5} (2\|r\|_6+ \|g\|_6) \\
& \lesssim \left( |\beta| \lambda^{-1} \|PU_{x,\lambda}\|_{6/5} + |\gamma| \|\partial_\lambda PU_{x,\lambda}\|_{6/5} \right) (\|r\|_6 + \|g\|_6) = o(\lambda^{-2}) \,.
\end{align*}
We have, since $r\in T_{x,\lambda}^\bot$ and $g\in T_{x,\lambda}$,
$$
\int_\Omega \nabla r\cdot\nabla g \,dy = 0 \,.
$$
This proves \eqref{eq:expfinaln2}.

\emph{Step 3.} We finally extract the relevant contribution from $r$ and show
\begin{equation}
\label{eq:expfinaln3}
\mathcal E_\epsilon[\psi_{x,\lambda}+r] = \mathcal E_\epsilon[\psi_{x,\lambda}] + \mathcal E_0[r] + o(\lambda^{-2}) + o(\epsilon\lambda^{-1}) \,.
\end{equation}

Indeed,
$$
\mathcal E_\epsilon[\psi_{x,\lambda}+r] = \mathcal E_\epsilon[\psi_{x,\lambda}] + 2\mathcal E_\epsilon[\psi_{x,\lambda},r] + \mathcal E_\epsilon[r] \,.
$$
Using $r\in T_{x,\lambda}^\bot$, the harmonicity of $H_0$ and equation \eqref{Ha-pde} for $H_a$, we find
$$
\int_\Omega \nabla\psi_{x,\lambda}\cdot\nabla r\,dy = - \lambda^{-1/2} \int_\Omega \nabla_y (H_a(x,y)-H_0(x,y))\cdot\nabla r\,dy = - \lambda^{-1/2} \int_\Omega a G_a(x,y)r\,dy \,.
$$
On the other hand, by \eqref{PU exp}, \eqref{sup-f} and \eqref{eq-d},
$$
\int_\Omega a\psi_{x,\lambda} r\,dy - \int_\Omega aU_{x,\lambda}r\,dy + \lambda^{-1/2} \int_\Omega aH_a(x,y)r\,dy = \mathcal O(\|a\|_{6/5} \|f_{x,\lambda}\|_\infty \|r\|_6 ) = o(\lambda^{-3}) \,.
$$
Thus,
$$
\mathcal E_0[\psi_{x,\lambda},r] = \int_\Omega a \left( U_{x,\lambda} - \lambda^{-1/2} H_a(x,y) - \lambda^{-1/2} G_a(x,y) \right) r\,dy + o(\lambda^{-3}) \,.
$$
By Lemma \ref{lem-V},
\begin{align*}
& \left| \int_\Omega a \left( U_{x,\lambda} - \lambda^{-1/2} H_a(x,y) - \lambda^{-1/2} G_a(x,y) \right) r\,dy \right| \\
& \quad \leq \|a\|_\infty \| U_{x,\lambda} - \lambda^{-1/2} H_a(x,\cdot) - \lambda^{-1/2} G_a(x,\cdot) \|_{6/5} \|r\|_6 = o(\lambda^{-5/2}) \,.
\end{align*}
Finally, by \eqref{eq:hbound} and \eqref{eq:psibound65},
$$
\left| \int_\Omega V\psi_{x,\lambda} r\,dy \right| \leq \|V\|_\infty \|\psi_{x,\lambda}\|_{6/5}\|r\|_6 = o(\lambda^{-1})
$$
and
$$
\left| \int_\Omega V r^2\,dy \right| \leq \|V\|_{3/2} \|r\|_6^2 = o(\lambda^{-1}) \,.
$$
This proves \eqref{eq:expfinaln3}.

The lemma follows by collecting the estimates from the three steps.
\end{proof}

%%%%%%%%%

\subsubsection*{Expanding the denominator}
\label{ssec expanding denominator}

Recall that $D_0$ and $\mathcal I[r]$ were defined in \eqref{eq:defnd} and \eqref{eq:defi} respectively. We shall show

\begin{lem}\label{expfinald}
As $\epsilon\to 0$,
$$
\alpha^{-6} \int_\Omega u_\epsilon^6\,dy = D_0 + D_1 + \mathcal I[r] + o(\lambda^{-2})\, ,
$$
where
$$
D_1 := 6 \int_\Omega \psi_{x,\lambda}^5 g\,dy + 15 \int_\Omega \psi_{x,\lambda}^4 g^2\,dy \,.
$$
\end{lem}

\begin{proof}
\emph{Step 1.} We show that the contribution from $h$ to $\alpha^{-6} \int_\Omega u_\epsilon^6\,dy$ is negligible, that is,
\begin{equation}
\label{eq:expfinald1}
\alpha^{-6} \int_\Omega u_\epsilon^6\,dy = \int_\Omega (\psi_{x,\lambda}+g+r)^6\,dy + o(\lambda^{-2}) \,.
\end{equation}

Indeed,
$$
\alpha^{-6} \int_\Omega u_\epsilon^6\,dy = \int_\Omega (\psi_{x,\lambda}+g+r)^6\,dy + 6 \int_\Omega (\psi_{x,\lambda}+g+r)^5 h \,dy + \mathcal O \left( \| \psi_{x,\lambda}+g+r \|_6^4 \|h\|_6^2 + \|h\|_6^6\right)
$$
and by \eqref{eq:hbound} the last term is $\mathcal O(\lambda^{-4})$. The middle term is
$$
\int_\Omega (\psi_{x,\lambda}+g+r)^5 h \,dy = \int_\Omega \psi_{x,\lambda}^5 h \,dy + \mathcal O \left( \| \psi_{x,\lambda}\|_6^4 \|g+r\|_6 \|h\|_6 + \|g+r\|_6^5 \|h\|_6 \right)
$$
and again by \eqref{eq:hbound} the last term here is $o(\lambda^{-5/2})$. The first term here is
$$
\int_\Omega \psi_{x,\lambda}^5 h \,dy = \int_\Omega U_{x,\lambda}^5 h \,dy + \mathcal O\left(\| U_{x,\lambda}\|_6^4 \|\psi_{x,\lambda}-U_{\lambda,x}\|_6 \|h\|_6 + \|\psi_{x,\lambda}-U_{x,\lambda}\|_6^5\|h\|_6 \right),
$$
which, by \eqref{eq:hbound} and \eqref{eq:psiuinfty}, is $\mathcal O(\lambda^{-5/2})$. Finally, by \eqref{eq:overlap} and \eqref{nabla-norm},
$$
\int_\Omega U_{x,\lambda}^5 h \,dy = 3^{-1} \int_\Omega \nabla PU_{x,\lambda}\cdot \nabla h\,dy = \sum_{j=1}^3 \delta_j \lambda^{-3} \int_\Omega \nabla PU_{x,\lambda}\cdot \nabla \partial_{x_j} PU_{x,\lambda}\,dy = \mathcal O(\lambda^{-3}) \,.
$$
This proves \eqref{eq:expfinald1}.

\emph{Step 2.} We now extract the relevant contribution from $g$ and show
\begin{equation}
\label{eq:expfinald2}
\int_\Omega (\psi_{x,\lambda}+g+r)^6\,dy = \int_\Omega (\psi_{x,\lambda}+r)^6\,dy + 6 \int_\Omega \psi_{x,\lambda}^5 g\,dy + 15\int_\Omega \psi_{x,\lambda}^4 g^2\,dy + o(\lambda^{-2}) \,.
\end{equation}

Indeed,
\begin{align*}
\int_\Omega (\psi_{x,\lambda}+g+r)^6\,dy & = \int_\Omega (\psi_{x,\lambda}+r)^6\,dy + 6 \int_\Omega (\psi_{x,\lambda}+r)^5 g \,dy + 15 \int_\Omega (\psi_{x,\lambda}+r)^4 g^2\,dy \\
& \quad + \mathcal O\left(\|\psi_{x,\lambda}+r\|_6^3 \|g\|_6^3 + \|g\|_6^6 \right)
\end{align*}
and by \eqref{eq:hbound} the last term is $\mathcal O(\lambda^{-3})$. We need to show that the contribution from $r$ to the second and third term on the right side is negligible. The third term is
$$
\int_\Omega (\psi_{x,\lambda}+r)^4 g^2\,dy = \int_\Omega \psi_{x,\lambda}^4 g^2\,dy + \mathcal O\left( \|\psi_{x,\lambda}\|_6^3 \|r\|_6 \|g\|_6^2 + \|r\|_6^4\|g\|_6^2 \right)
$$
and by \eqref{eq:hbound} the last term is $o(\lambda^{-5/2})$. The second term above is
$$
\int_\Omega (\psi_{x,\lambda}+r)^5 g \,dy = \int_\Omega \psi_{x,\lambda}^5 g \,dy + 5 \int_\Omega \psi_{x,\lambda}^4 r g \,dy + \mathcal O\left( \|\psi_{x,\lambda}\|_6^3 \|r\|_6^2 \|g\|_6 + \|r\|_6^5\|g\|_6^2 \right)
$$
and by \eqref{eq:hbound} the last term is $o(\lambda^{-2})$. Let us show that the second term on the right side of the previous equation is negligible. We have
$$
\int_\Omega \psi_{x,\lambda}^4 r g \,dy = \int_\Omega U_{x,\lambda}^4 r g \,dy + \mathcal O\left( \|U_{x,\lambda}\|_6^3 \|\psi_{x,\lambda}-U_{x,\lambda}\|_6 \|r\|_6\|g\|_6 + \|\psi_{x,\lambda}-U_{x,\lambda}\|_6^4 \|r\|_6 \|g\|_6 \right)
$$
and by \eqref{eq:hbound} and \eqref{eq:psiuinfty} the last term is $o(\lambda^{-2})$. Now
\begin{align*}
\int_\Omega U_{x,\lambda}^4 r g \,dy & = \beta \lambda^{-1} \int_\Omega U_{x,\lambda}^4 PU_{x,\lambda} r\,dy + \gamma \int_\Omega U_{x,\lambda}^4 \partial_\lambda PU_{x,\lambda} r\,dy \\
& = \beta \lambda^{-1} \int_\Omega U_{x,\lambda}^5 r\,dy + \gamma \int_\Omega U_{x,\lambda}^4 \partial_\lambda U_{x,\lambda} r\,dy \\
& \quad + \mathcal O( (|\beta| \lambda^{-1} \| PU_{x,\lambda} - U_{x,\lambda}\|_6 + |\gamma|\|\partial_\lambda PU_{x,\lambda} - \partial_\lambda U_{x,\lambda}\|_6) \|U_{x,\lambda}\|_6^4 \|r\|_6 ) \,.
\end{align*}
By Lemma \ref{lem-bg}, \cite[Prop.\! 1 (c)]{rey2} and \eqref{eq:hbound}, the last term is $o(\lambda^{-2})$. Finally, by \eqref{eq:laplzeromodes1} and the fact that $r\in T_{x,\lambda}^\bot$,
$$
\int_\Omega U_{x,\lambda}^5 r\,dy = 3^{-1} \int_\Omega \nabla PU_{x,\lambda}\cdot\nabla r\,dy = 0 \,,
\quad
\int_\Omega U_{x,\lambda}^4 \partial_\lambda U_{x,\lambda} r\,dy = (15)^{-1} \int_\Omega \nabla \partial_\lambda PU_{x,\lambda}\cdot\nabla r\,dy = 0 \,.
$$
This proves \eqref{eq:expfinald2}.

\emph{Step 3.} We finally extract the relevant contribution from $r$ and show
\begin{equation}
\label{eq:expfinald3}
\int_\Omega (\psi_{x,\lambda}+r)^6\,dy = \int_\Omega \psi_{x,\lambda}^6\,dy + \mathcal I[r] + o(\lambda^{-2}) \,.
\end{equation} 

Indeed,
\begin{align*}
\int_\Omega (\psi_{x,\lambda}+r)^6\,dy & = \int_\Omega \psi_{x,\lambda}^6\,dy + 6 \int_\Omega \psi_{x,\lambda}^5 r\,dy + 15 \int_\Omega \psi_{x,\lambda}^4 r^2\,dy + 20 \int_\Omega \psi_{x,\lambda}^3 r^3 \,dy \\
& \quad + \mathcal O\left( \|\psi_{x,\lambda}\|_6^2 \|r\|_6^4 + \|r\|_6^6 \right)
\end{align*}
and by \eqref{eq:hbound} the last term is $o(\lambda^{-2})$. We need to extract $\mathcal I[r]$ from the three terms on the right side involving $r$. We begin with the term which is linear in $r$,
\begin{align*}
\int_\Omega \psi_{x,\lambda}^5 r\,dy & = \int_\Omega U_{x,\lambda}^5 r\,dy + 5 \int_\Omega U_{x,\lambda}^4(\psi_{x,\lambda}-U_{x,\lambda})r\,dy \\
& \quad + \mathcal O \left( \|U_{x,\lambda}\|_{18/5}^3 \|\psi_{x,\lambda}-U_{x,\lambda}\|_\infty \|r\|_6 + \| \psi_{x,\lambda}-U_{x,\lambda}\|_6^5 \|r\|_6 \right).
\end{align*}
By \eqref{eq:psiuinfty}, \eqref{eq:hbound} and $\|U_{x,\lambda}\|_{18/5}^3 = \mathcal O(\lambda^{-1})$, the last term is $o(\lambda^{-2})$. Since $r\in T_{x,\lambda}^\bot$, the first term is
$$
\int_\Omega U_{x,\lambda}^5 r\,dy = 3^{-1} \int_\Omega \nabla PU_{x,\lambda}\cdot\nabla r\,dy = 0 \,.
$$
Writing $\psi_{x,\lambda}- U_{x,\lambda}= - \lambda^{-1/2} H_a(x,\cdot) - f_{x,\lambda}$, we have
\begin{align*}
& \int_\Omega U_{x,\lambda}^4 (\psi_{x,\lambda}-U_{x,\lambda})r\,dy = - \lambda^{-1/2} \int_\Omega U_{x,\lambda}^4 H_a(x,y)r\,dy + \mathcal O(\| U_{x,\lambda}\|_{24/5}^4 \|f_{x,\lambda}\|_\infty \|r\|_6 ) \,.
\end{align*}
By \eqref{sup-f}, \eqref{eq-d}, \eqref{eq:hbound} and $\|U_{x,\lambda}\|_{24/5}^4=\mathcal O(\lambda^{-1/2})$, the last term on the right side is $o(\lambda^{-2})$.

We now turn to the terms that are quadratic in $r$. We have
$$
\int_\Omega \psi_{x,\lambda}^4 r^2\,dy = \int_\Omega U_{x,\lambda}^4 r^2\,dy + \mathcal O\left( \|U_{x,\lambda}\|_{9/2}^3 \|\psi_{x,\lambda}-U_{x,\lambda}\|_\infty \|r\|_6^2 + \|\psi_{x,\lambda}-U_{x,\lambda}\|_6^4 \|r\|_6^2 \right)
$$
and by \eqref{eq:psiuinfty}, \eqref{eq:hbound} and $\|U_{x,\lambda}\|_{9/2}^3 = \mathcal O(\lambda^{-1/2})$, the last term on the right side is $o(\lambda^{-2})$. Similarly, one shows that
$$
\int_\Omega \psi_{x,\lambda}^3 r^3\,dy = \int_\Omega U_{x,\lambda}^3 r^3\,dy + o(\lambda^{-2}) \,.
$$
This proves \eqref{eq:expfinald3}.

The lemma follows by collecting the estimates from the three steps.
\end{proof}

%%%%%%%%%

\begin{proof}[Proof of Lemma \ref{expfinalq}]
Note that, by \eqref{eq:hbound}, $D_1=\mathcal O(\lambda^{-1})$ and $\mathcal I[r]=o(\lambda^{-1})$. Moreover, by \eqref{eq:upperd}, $D_0$ stays away from zero. Therefore, the expansion from Lemma \ref{expfinald} implies that
$$
\left( \alpha^{-6} \int_\Omega u_\epsilon^6\,dy \right)^{-1/3}
= D_0^{-1/3} \left( 1 - \frac13 \frac{D_1}{D_0} - \frac13 \frac{\mathcal I[r]}{D_0} + \frac{2}{9} \frac{D_1^2}{D_0^2} + o(\lambda^{-2}) \right).
$$
Combining this with the expansion from Lemma \ref{expfinaln} and using $N_1=\mathcal O(\lambda^{-1})$ (again from \eqref{eq:hbound}), we obtain
\begin{align*}
\mathcal S_{a+\epsilon V}[u_\epsilon] & = \mathcal S_{a+\epsilon V}[\psi_{x,\lambda}] + A + D_0^{-1/3}\left( \mathcal E_0[r] - \frac{N_0}{3\,D_0}\mathcal I[r] \right)
+ o(\lambda^{-2}) + o(\epsilon \lambda^{-1})
\end{align*}
with
$$
A = D_0^{-1/3} \left( N_1 - \frac{D_1}{3\,D_0} N_1 - \frac{D_1}{3D_0} N_0 + \frac{2}{9} \frac{D_1^2}{D_0^2} N_0 \right).
$$
Thus, the assertion of the lemma is equivalent to $A=o(\lambda^{-2})+o(\epsilon\lambda^{-1})$. We write
\begin{align*}
A = D_0^{-1/3} \left( \left(N_1- D_1 \right)\left(1 - \frac{D_1}{3\,D_0} \right) + \frac{1}{3} \frac{D_1^2}{D_0} + \left( 1 - \frac{N_0}{3D_0} \right) D_1 \left(1 - \frac{2\,D_1}{3\,D_0} \right) \right).
\end{align*}
It follows from \eqref{eq:uppern} and \eqref{eq:upperd} that
\begin{equation}
\label{eq:nodo}
\frac{N_0}{3\,D_0} = 1 + \mathcal O(\lambda^{-2}) + \mathcal O(\epsilon\lambda^{-1}) \,.
\end{equation}
This, together with $D_1=\mathcal O(\lambda^{-1})$, yields
$$
A = D_0^{-1/3} \left( \left(N_1- D_1 \right)\left(1 - \frac{D_1}{3\,D_0} \right) + \frac{1}{3} \frac{D_1^2}{D_0} \right) + o(\lambda^{-2}).
$$

We shall show in Appendix \ref{sec-app-a} that
\begin{equation}
\label{eq:n1}
N_1 = \frac{3\pi^2}{2}\beta\, \lambda^{-1} + \left( \frac{3\pi^2}{4}\beta^2 + \frac{15\,\pi^2}{64} \gamma^2 - 8\pi\,\phi_0(x)\,\beta + 4\pi\,\phi_0(x)\,\gamma\right)\lambda^{-2} + o(\lambda^{-2})
\end{equation}
and
\begin{equation}
\label{eq:d1}
D_1 = \frac{3\pi^2}{2}\beta \lambda^{-1} + \left( \frac{15\,\pi^2}{4} \beta^2 + \frac{15\,\pi^2}{64} \gamma^2 - 8\pi\,\phi_0(x)\,\beta +4\pi\,\phi_0(x)\,\gamma \right)\lambda^{-2} + o(\lambda^{-2}) \,.
\end{equation}
Thus, in particular,
$$
N_1 -D_1 = -3\pi^2 \beta^2 \lambda^{-2} + o(\lambda^{-2})
\qquad\text{and}\qquad
D_1^2 = \left( \frac{3\pi^2}{2} \right)^2 \beta^2\lambda^{-2} + o(\lambda^{-2}) \,.
$$
This, together with $D_0 = (S/3)^{3/2} + o(\lambda^{-1})$ (from \eqref{eq:upperd}), implies $A=o(\lambda^{-2})$, as claimed.
\end{proof}

Before continuing with the main line of the argument, let us expand $\alpha$. By the normalization \eqref{appr-min}, Lemma \ref{expfinald}, \eqref{eq:upperd} and \eqref{eq:d1}
\begin{align}\label{eq:expalpha}
\alpha^{-6} (S/3)^{3/2} & = (S/3)^{3/2} + \frac{3\pi^2}{2}\,\beta\,\lambda^{-1} - 8\pi\,\phi_a(x)\,\lambda^{-1} \notag \\
& \quad  + \left(8\pi \,a(x) + \frac{15\,\pi^2}{4} \beta^2 + \frac{15\,\pi^2}{64} \gamma^2 - 8\pi\,\phi_0(x)\,\beta +4\pi\,\phi_0(x)\,\gamma \right)\lambda^{-2} +
\mathcal I[r] + o(\lambda^{-2}) \,.
\end{align}

%%%%%%%%%%%%%%%

\subsection{Coercivity}

To complete the proof of our main results, it remains to prove that the terms involving $r$ in the expansion \eqref{eq:expfinalq} give a non-negative contribution. Recall that $\mathcal I[r]$ was defined in \eqref{eq:defi} and $N_0$ and $D_0$ in Lemmas \ref{expfinaln} and \ref{expfinald}, respectively.

\begin{lem} \label{lem-g}
There is a $\rho>0$ such that for all sufficiently small $\epsilon>0$,
$$
\mathcal E_0[r] - \frac{N_0}{3\,D_0}\mathcal I[r] \geq \rho \int_\Omega |\nabla r|^2\,dy + o(\lambda^{-2}) \,.
$$
\end{lem}

\begin{proof}
We bound, using \eqref{eq-d}, Lemma \ref{lem-uh2} and \eqref{eq:concpoint}, for any $\delta>0$,
\begin{align*}
\left| 30\,\lambda^{-1/2} \int_\Omega U_{x,\lambda}^4 H_a(x,y) r\,dy \right|
& \leq 30 \, \lambda^{-1/2} \left(  \int_\Omega U_{x,\lambda}^4 r^2\,dy \right)^{\frac 12} \left( \int_\Omega U_{x, \lambda}^4 \, H_a(x,y)^2\,dy\right)^{\frac 12} \nonumber  \\
& \leq o(\lambda^{-1}) \left(  \int_\Omega U_{x,\lambda}^4 r^2\,dy \right)^{\frac 12}
\leq \delta \int_\Omega U_{x,\lambda}^4 \, r^2\,dy + \delta^{-1} o(\lambda^{-2}) \,.
\end{align*}
Similarly, using \eqref{eq:hbound},
\begin{align*}
\left| 20 \int_\Omega U_{x,\lambda}^3 \,  r^3\,dy  \right| &  \leq 20 \left(  \int_\Omega U_{x,\lambda}^4 \, r^2\,dy \right)^{\frac 34}  \left( \int_\Omega r^6 \,dy \right)^{\frac 14} \leq o(\lambda^{-\frac34 }) \left(  \int_\Omega U_{x,\lambda}^4 \, r^2\,dy \right)^{\frac 34} \\
& \leq \delta  \int_\Omega U_{x,\lambda}^4 \, r^2\,dy + \delta^{-3} \, o(\lambda^{-3}) \,.
\end{align*}
This, together with \eqref{eq:nodo} implies that
\begin{align*}
\mathcal E_0[r] - \frac{N_0}{3\,D_0}\mathcal I[r] & \geq \int_\Omega \left( |\nabla r|^2 + a r^2 - 15\, U_{x,\lambda}^4 r^2\right)dy \\
& \quad - \left( 2\delta + \mathcal O(\lambda^{-2}) + \mathcal O(\epsilon \lambda^{-1}) \right) \int_\Omega U_{x,\lambda}^4 \, r^2\,dy + \delta^{-1} o(\lambda^{-2}) + \delta ^{-3} o(\lambda^{-3}) \,.
\end{align*}
Since $r\in T_{x,\lambda}^\bot$, Lemma \ref{coercivity} implies that for all sufficiently small $\epsilon>0$, the first term on the right side is bounded from below by $\rho \int_\Omega |\nabla r|^2\,dy$ for some $\rho>0$ independent of $\epsilon$. On the other hand, by \eqref{hs-ineq}, choosing $\delta>0$ small, but independent of $\epsilon$, and then $\epsilon$ small, we can make sure that
$$
- \left( 2\delta + \mathcal O(\lambda^{-2}) + \mathcal O(\epsilon \lambda^{-1}) \right) \int_\Omega U_{x,\lambda}^4 \, r^2\,dy \geq - (\rho/2) \int_\Omega |\nabla r|^2\,dy \,.
$$
This completes the proof of the lemma.
\end{proof}

%%%%%%%%%%%%%%%

\subsection{Proof of the main results}

In this subsection we prove Theorems \ref{thm-main}, \ref{thm-main2} and \ref{thm-minimizers}. Combining the expansions from Lemma \ref{expfinalq} and Theorem \ref{exppsi} and using the fact that $\phi_a(x_0)=0$ (see Proposition \ref{thm-q}) we obtain
\begin{align*}
\mathcal S_{a+\epsilon V}[u_\epsilon] & \geq S + (S/3)^{-1/2} \left( \frac{\epsilon}{\lambda} Q_V(x_0) - \frac{2\pi^2 a(x_0)}{\lambda^2} \right) \notag \\
& \quad + (S/3)^{-1/2} 4\pi\, \phi_a(x)\, \lambda^{-1}  + (S/3)^{-1/2} \left( \mathcal E_0[r] - \frac{N_0}{3\, D_0} \mathcal I[r] \right) + o(\lambda^2) + o(\epsilon\lambda^{-1}) \,.
\end{align*}
Using the almost minimizing assumption \eqref{appr-min} as well as the coercivity bound from Lemma \ref{lem-g} we obtain
\begin{align}\label{eq:proofmain}
0 & \geq (1+o(1)) (S-S(a+\epsilon V)) + (S/3)^{-1/2} \left( \frac{\epsilon}{\lambda} Q_V(x_0) - \frac{2\pi^2 a(x_0)}{\lambda^2} \right) + \mathcal R + o(\lambda^2) + o(\epsilon\lambda^{-1}) \,.
\end{align}
with
\begin{equation}  \label{R}
\mathcal R := (S/3)^{-1/2} \left( 4\pi \phi_a(x) \lambda^{-1} + \rho \int_\Omega |\nabla r|^2\,dy \right).
\end{equation}
Note that, by Corollary \ref{cor-2}, $\mathcal R\geq 0$.

%%%%%%%

\begin{lem}
If $\mathcal N_a(V)\neq\emptyset$, then $x_0\in\mathcal N_a(V)$.
\end{lem}

This is the only place in the proof of Theorem \ref{thm-main} where we need assumption \eqref{eq:anegative}.

\begin{proof}
We recall the upper bound from Corollary \ref{thm-upperb},
$$
S(a+\epsilon V) \leq S - (S/3)^{-1/2} \sup_{y\in\mathcal N_a(V)} \frac{Q_V(y)^2}{8\pi^2|a(y)|}\ \epsilon^2 + o(\epsilon^2) \,.
$$
Combining this with \eqref{eq:proofmain} and using $\mathcal R\geq 0$, we find
$$
C_1 \,\epsilon^2 + C_2 \,\lambda^{-2} \leq  \left( -(S/3)^{-1/2} Q_V(x_0) + o(1) \right) \frac{\epsilon}{\lambda}
$$
with
$$
C_1 := (S/3)^{-1/2} \sup_{y\in\mathcal N_a(V)} \frac{Q_V(y)^2}{8\pi^2|a(y)|} + o(1) \,,
\qquad
C_2 := (S/3)^{-1/2}\,  2\pi^2 |a(x_0)| + o(1) \,.
$$
By the assumptions $\mathcal N_a(V)\neq\emptyset$ and \eqref{eq:anegative}, both $C_1$ and $C_2$ tend to some positive quantities as $\epsilon\to 0$. Since $C_1\epsilon^2 + C_2\lambda^{-2}\geq 2\sqrt{C_1\,C_2} \, \epsilon\lambda^{-1}$ we obtain that $Q_V(x_0)<0$, as claimed.
\end{proof}

We now assume $\mathcal N_a(V)\neq\emptyset$ and complete the proof of Theorems \ref{thm-main} and \ref{thm-minimizers}. We can write
\begin{align*}
& (S/3)^{-1/2} \left( \frac{\epsilon}{\lambda} Q_V(x_0) - \frac{2\pi^2 a(x_0)}{\lambda^2} \right) + o(\lambda^2) + o(\epsilon\lambda^{-1}) = - (S/3)^{-1/2} \frac{\left( Q_V(x_0)+o(1) \right)^2}{4\left(2\pi^2|a(x_0)|+o(1)\right)} \ \epsilon^2 \\
& \qquad\qquad\qquad\qquad\qquad\qquad + (S/3)^{-1/2} \left( \frac{Q_V(x_0)+o(1)}{2\sqrt{2\pi^2 |a(x_0)|+o(1)}} \ \epsilon + \sqrt{2\pi^2|a(x_0)|+ o(1)}\ \lambda^{-1} \right)^2.
\end{align*}
Inserting this into \eqref{eq:proofmain} we obtain
\begin{equation}
\label{eq:proofmain2}
(S/3)^{-1/2} \frac{\left( Q_V(x_0)+o(1) \right)^2}{4\left(2\pi^2|a(x_0)|+o(1)\right)} \ \epsilon^2
\geq (1+ o(1)) \left( S-S(a+\epsilon V) \right) + \mathcal R'
\end{equation}
with
\begin{align} \label{R'}
\mathcal R' & := \mathcal R + (S/3)^{-1/2} \left( \frac{Q_V(x_0)+o(1)}{2\sqrt{2\pi^2 |a(x_0)|+o(1)}} \ \epsilon + \sqrt{2\pi^2|a(x_0)|+ o(1)}\ \lambda^{-1} \right)^2 \,.
\end{align}
Since $\mathcal R' \geq 0$ we obtain, in particular,
\begin{align}
\label{eq:proofmain3}
S-S(a+\epsilon V) & \leq (1+o(1)) (S/3)^{-1/2} \frac{\left( Q_V(x_0)+o(1) \right)^2}{4\left(2\pi^2|a(x_0)|+o(1)\right)} \ \epsilon^2 = (S/3)^{-1/2} \frac{Q_V(x_0)^2}{8\pi^2|a(x_0)|} \ \epsilon^2 + o(\epsilon^2) \notag \\
& \leq (S/3)^{-1/2} \sup_{y\in\mathcal N_a(V)} \frac{Q_V(y)^2}{8\pi^2|a(y)|}\,\epsilon^2 + o(\epsilon^2) \,.
\end{align}
In the last inequality we used $x_0\in\mathcal N_a(V)$. This proves the claimed lower bound on $S(a+\epsilon V)$ and completes the proof of Theorem \ref{thm-main}.

\medskip

We now proceed to the proof of Theorem \ref{thm-minimizers}, still under the assumption $\mathcal N_a(V)\neq\emptyset$. Combining the lower bound on $S-S(a+\epsilon V)$ from Corollary \ref{thm-upperb} with the upper bound in \eqref{eq:proofmain3} we obtain
$$
\frac{Q_V(x_0)^2}{|a(x_0)|} = \sup_{y\in\mathcal N_a(V)} \frac{Q_V(y)^2}{|a(y)|} \,.
$$
Moreover, inserting the lower bound on $S-S(a+\epsilon V)$ into \eqref{eq:proofmain2} we infer that $\mathcal R' = o(\epsilon^2)$. Thus, by \eqref{R} and \eqref{R'} 
$$
\|\nabla r\|^2 = o(\epsilon^2)
\qquad\text{and}\qquad
\lambda^{-1} = \frac{|Q_V(x_0)|}{4\pi^2\, |a(x_0)|}\, \epsilon + o(\epsilon) \,.
$$
and, reinserting the last expression into $\mathcal R = o(\epsilon^2)$, also
$$
\phi_a(x) = o(\epsilon) \,.
$$
Inserting these bounds into \eqref{eq:expalpha}, we obtain
\begin{align*}
\alpha^{-6} & = 1 + (S/3)^{-3/2}\, \frac{3\pi}{2}\,\beta\,\lambda^{-1} \\
& \quad + (S/3)^{-3/2} \Big( 8\pi\,a(x_0) + \frac{15 \pi^2}{4}\,\beta^2 + \frac{15\pi^2}{64}\,\gamma^2 - 8\pi\,\phi_0(x_0)\,\beta + 4\pi\,\phi_0(x_0)\,\gamma \Big) \lambda^{-2} + o(\epsilon^2)
\end{align*}
and therefore, using Lemma \ref{lem-bg}, $\alpha=1+ \mathcal O(\epsilon)$. This completes the proof of Theorem \ref{thm-minimizers}.

\medskip

We now assume $\mathcal N_a(V)=\emptyset$ and prove Theorem \ref{thm-main2}. Estimating $Q_V(x_0)\geq 0$ and $\mathcal R\geq 0$ in \eqref{eq:proofmain} we obtain
$$
0 \geq (1+o(1)) (S-S(a+\epsilon V)) + \left( (S/3)^{-1/2}\,  2\pi^2 |a(x_0)| + o(1) \right) \lambda^{-2} + o(\epsilon\lambda^{-1}) \,.
$$
Since $o(\epsilon \lambda^{-1}) \geq - \delta \lambda^{-2} + o(\epsilon^2)$ for any fixed $\delta$, this implies $S-S(a+\epsilon V)= o(\epsilon^2)$. 

Under the additional assumption $Q_V(x_0)>0$, we infer from \eqref{eq:proofmain} that
$$
0 \geq (1+o(1)) (S-S(a+\epsilon V)) + C_1 \epsilon\lambda^{-1} + C_2 \lambda^{-2}
$$
with
$$
C_1 := (S/3)^{-1/2} \, Q_V(x_0) + o(1)
\qquad\text{and}\qquad
C_2 := (S/3)^{-1/2}\,  2\pi^2 |a(x_0)| + o(1) \,.
$$
Since both $C_1$ and $C_2$ are positive for all sufficiently small $\epsilon>0$, we arrive at a contradiction. Thus, assumption \eqref{appr-min0}, under which we have worked so far, is not satisfied. By the concavity argument in the proof of Corollary \ref{thm-upperb} this means that $S(a+\epsilon V)=S$ for all sufficiently small $\epsilon>0$. This concludes the proof of Theorem \ref{thm-main2}.

%%%%%%%%%%%%%%%%%%%%%%%%%%%%%%%%%%%%%%%%%%%%%%%%%%%%%%%%%%%%%%%%%%%%%%%
\appendix
 
\section{Some computations}
\label{sec-app-a}

\subsection{Asymptotics and bounds}

We recall that we abbreviate $d=\mathrm{dist}(x,\partial\Omega)$.

\begin{lem}
As $\lambda\to\infty$, uniformly in $x\in\Omega$,
\begin{equation}\label{eq:nablapu}
\int_\Omega |\nabla P_{x,\lambda}|^2\,dy = 3^{-1/2} S^{3/2} - 4\pi\, \phi_0(x)\, \lambda^{-1} + o((\lambda d)^{-1}) \,,
\end{equation}
\begin{equation}
\label{eq:pu6}
\int_\Omega PU_{x,\lambda}^6\,dy = (S/3)^{3/2} - 8\pi\, \phi_0(x)\, \lambda^{-1} + o((\lambda d)^{-1}) \,.
\end{equation}
\end{lem}

\begin{proof}
We set again $\phi_{x,\lambda}=U_{x,\lambda}-PU_{x,\lambda}$. Then, by \eqref{eq-pu} and \eqref{el},
$$
\int_\Omega |\nabla P_{x,\lambda}|^2\,dy = \int_\Omega \nabla PU_{x,\lambda} \cdot \nabla U_{x,\lambda} \,dy = 3 \int_\Omega PU_{x,\lambda} U_{x,\lambda}^5 \,dy = 3\int_\Omega U_{x,\lambda}^6\,dy - 3\int_\Omega U_{x,\lambda}^5\phi_{x,\lambda}\,dy \,.
$$
By \cite[Proof of (B.3)]{rey2}
\begin{equation}
\label{eq:u6}
\int_\Omega U_{x,\lambda}^6\,dy = (S/3)^{3/2} + o((d\lambda)^{-1}) \,,
\end{equation}
and, as shown in \cite[Proof of Thm.\!\! 1.1]{esp},
\begin{equation}
\label{eq:u5phi}
\int_\Omega U_{x,\lambda}^5\phi_{x,\lambda} \,dy = \frac{4\pi}{3}\, \phi_0(x)\, \lambda^{-1} + o( (d\lambda)^{-1}) \,.
\end{equation}
(Since $\phi_{x,\lambda}=\lambda^{-1/2} H_0(x,\cdot)+f_{x,\lambda}$, the proof of the latter relation is similar to the proof of Lemma \ref{lem-uh}, but to get the uniformity even for $x$ close to the boundary more careful bounds on $\nabla_y H_0(x,y)$ are needed.) This proves \eqref{eq:nablapu}.

To prove \eqref{eq:pu6}, we write
$$
\int_\Omega PU_{x,\lambda}^6\,dy = \int_\Omega U_{x,\lambda}^6\,dy - 6 \int_\Omega U_{x,\lambda}^5 \phi_{x,\lambda}\,dy + \mathcal O\left( \|U_{x,\lambda}\|_4^4 \|\phi_{x,\lambda}\|_\infty^2 + \|\phi_{x,\lambda}\|_6^6 \right).
$$
For the first two terms we use \eqref{eq:u6}, \eqref{eq:u5phi}. Moreover, $\|\phi_{x,\lambda}\|_\infty = \mathcal O(\lambda^{-1/2} d^{-1})$ (from \eqref{sup-h} and \eqref{sup-f}), $\|\phi_{x,\lambda}\|_6 = \mathcal O((d\lambda)^{-1/2})$ (from \cite[Prop.\!\! 1 (c)]{rey2}) and $\|U_{x,\lambda}\|_4^4=\mathcal O(\lambda^{-1})$, so the remainder term is $o((d\lambda)^{-1})$.
\end{proof}

\begin{lem}
As $\lambda\to\infty$, uniformly in $x\in\Omega$,
 \begin{equation}
\label{eq:zm65}
\| PU_{x,\lambda} \|_{6/5} = \mathcal O(\lambda^{-1/2}) \,.
\end{equation}
Moreover, for $x$ in compact subset of $\Omega$,
 \begin{equation}
\label{eq:zm65der}
\| \partial_\lambda PU_{x,\lambda} \|_{6/5} = \mathcal O(\lambda^{-3/2}) \,.
\end{equation}
\begin{equation}
\label{eq:psiuinfty}
\| \psi_{x,\lambda}-U_{x,\lambda}\|_\infty = \mathcal O(\lambda^{-1/2}) \,.
\end{equation}
and
\begin{equation}
\label{eq:psibound65}
\|\psi_{x,\lambda}\|_{6/5} = \mathcal O(\lambda^{-1/2}) \,.
\end{equation}
\end{lem}

\begin{proof}
The bound \eqref{eq:zm65} follows from $0\leq PU_{x,\lambda}\leq U_{x,\lambda}$ (see \cite[Prop.\! 1(a)]{rey2}) and a straightforward computation for $U_{x,\lambda}$, using the fact that $\Omega$ is bounded.

To prove \eqref{eq:zm65der} we first note that, by a straightforward computation, the claimed bound holds with $\partial_\lambda U_{x,\lambda}$ instead of $\partial_\lambda PU_{x,\lambda}$. The claimed bound now follows since by the bound on $\partial_\lambda U_{x,\lambda}-\partial_\lambda PU_{x,\lambda}$ in \cite[Prop.\! 1 (c)]{rey2} (which holds even in $L^6$).

For the proof of \eqref{eq:psiuinfty} we write $\psi_{x,\lambda}-U_{x,\lambda}= -\lambda^{-1/2}H_a(x,\cdot)-f_{x,\lambda}$. Then \eqref{eq:psiuinfty} follows from  \eqref{sup-h-2} and \eqref{sup-f}. Finally, \eqref{eq:psibound65} follows from \eqref{eq:zm65} and \eqref{eq:psiuinfty}.
\end{proof}

%%%%%%%%%%%%%

\subsection{Proof of \eqref{eq:n1}}

We have
\begin{align*}
N_1 & = \beta^2 \lambda^{-2} \int_\Omega |\nabla PU_{x,\lambda}|^2\,dy 
+ \gamma^2 \int_\Omega |\nabla\partial_\lambda PU_{x,\lambda}|^2\,dy
+2 \beta\gamma \lambda^{-1} \int_\Omega \nabla PU_{x,\lambda}\cdot\nabla\partial_\lambda PU_{x,\lambda}\,dy \\
& \quad + 2\beta \lambda^{-1} \int_\Omega \nabla\psi_{x,\lambda}\cdot\nabla PU_{x,\lambda}\,dy + 2\gamma \int_\Omega \nabla\psi_{x,\lambda}\cdot\nabla\partial_\lambda PU_{x,\lambda}\,dy \\
& \quad + 2\beta \lambda^{-1} \int_\Omega a \psi_{x,\lambda} PU_{x,\lambda}\,dy + 2\gamma \int_\Omega a \psi_{x,\lambda} \partial_\lambda PU_{x,\lambda}\,dy \,.
\end{align*}
Therefore \eqref{eq:n1} will follow from the following relations, together with the facts that $\phi_a(x)=o(1)$ by Proposition \ref{thm-q} and that $\beta,\gamma=\mathcal O(1)$ by Lemma \ref{lem-bg},
\begin{align}
\lambda^{-2} \int_\Omega |\nabla PU_{x,\lambda}|^2 \,dy &= \frac{3 \pi^2}{4} \lambda^{-2}  + o (\lambda^{-2}), \label{num1} \\
\int_\Omega |\nabla \partial_\lambda PU_{x,\lambda}|^2 \,dy &=  \frac{15\,\pi^2}{64 } \lambda^{-2} + o(\lambda^{-2}), \label{num2} \\
\lambda^{-1} \int_\Omega \nabla PU_{x,\lambda}\cdot\nabla\partial_\lambda PU_{x,\lambda} \,dy &= o(\lambda^{-2}), \label{num2a} \\
 \lambda^{-1} \int_\Omega \nabla \psi_{x,\lambda} \cdot \nabla PU_{x,\lambda} \,dy &=  \frac{3 \pi^2}{4}  \lambda^{-1} - 4 \pi\,  \phi_a(x)\, \lambda^{-2} + o (\lambda^{-2}), \label{num3} \\
 \int_\Omega \nabla \psi_{x,\lambda} \cdot \nabla \partial_\lambda PU_{x,\lambda} \,dy &=  2 \pi\,  \phi_a(x)\, \lambda^{-2}   + o(\lambda^{-2}), \label{num5} \\
 \lambda^{-1} \int_\Omega a \psi_{x,\lambda} PU_{x,\lambda} \,dy &=  4 \pi\, (\phi_a(x) - \phi_0(x))\, \lambda^{-2} + o (\lambda^{-2}), \label{num4}\\
 \int_\Omega a \psi_{x,\lambda} \partial_\lambda PU_{x,\lambda} \,dy 
&= - 2 \pi\,  (\phi_a(x) - \phi_0(x))\, \lambda^{-2}  + o (\lambda^{-2}) . \label{num6}
\end{align}

For the proof of these bounds we recall that $d\gtrsim 1$ by Proposition \ref{wbound}.

The bounds \eqref{num1}, \eqref{num2} and \eqref{num2a} follow from \cite[(B.2), (B.7) and (B.5)]{rey2}, respectively.

For the proof of the remaining assertions we decompose $\psi_{x,\lambda}=U_{x,\lambda}-\lambda^{-1/2} H_a(x,\cdot) - f_{x,\lambda}$ and recall the bound \eqref{sup-f} on $f_{x,\lambda}$.

\emph{Proof of \eqref{num3}.   } 
By \eqref{eq-pu} and \eqref{el},
$$
\lambda^{-1} \int_\Omega\nabla \psi_{x,\lambda} \cdot \nabla PU_{x,\lambda}\,dy = 3 \lambda^{-1} \int_\Omega U_{x,\lambda}^5 (U_{x,\lambda} - \lambda^{-1/2}H_a(x,\cdot))\,dy + o(\lambda^{-2}). 
$$
By \eqref{eq:u6}, $3 \lambda^{-1} \int_\Omega  U_{x,\lambda}^6\,dy = \frac{3 \pi^2 }{4} \lambda^{-1}+o(\lambda^{-2})$. On the other hand, by Lemma \ref{lem-uh},
\[ 
3 \lambda^{-3/2} \int_\Omega U_{x,\lambda}^5 H_a(x, y)\,dy = 4 \pi \,\phi_a(x)\, \lambda^{-2} + o(\lambda^{-2}). 
\] 

\emph{Proof of \eqref{num5}.   } 
By differentiating \eqref{eq-pu} and \eqref{el},
\[ \int_\Omega\nabla \psi_{x,\lambda} \cdot \nabla \partial_\lambda PU_{x,\lambda}\,dy = 15 \int_\Omega (U_{x,\lambda} - \lambda^{-1/2} H_a(x, y))  U_{x,\lambda}^4 \partial_\lambda U_{x,\lambda}\,dy + o(\lambda^{-2}). \]
To compute the first summand, we use $\int_{\R^3}  U_{x,\lambda}^5 \partial_\lambda U_{x,\lambda}\,dy = \partial_\lambda  \int_{\R^3}  U_{x,\lambda}^6\,dy = 0$ and thus
\[  \left| \int_\Omega U_{x,\lambda}^5 \partial_\lambda U_{x,\lambda}\,dy \right| =  \left| \int_{\R^3 \setminus \Omega } U_{x,\lambda}^5 \partial_\lambda U_{x,\lambda}\,dy \right| \leq (2\lambda)^{-1} \int_{\R^3 \setminus B_{\lambda d}(x)} \frac{|1-|x-z|^2|}{(1 + |x-z|^2)^{4}}\,dz = \mathcal{O}(\lambda^{-4}). \]

To compute the second summand we argue similarly as in the proof of Lemmas \ref{lem-uh} and \ref{lem-uh2} and obtain
 \[ -15\, \lambda^{-1/2} \int_\Omega H_a(x,y)  U_{x,\lambda}^4 \partial_\lambda U_{x,\lambda}\,dy = 2 \pi\, \phi_a(x)\,\lambda^{-2} + o(\lambda^{-2}) \,. \]
The constant comes from
$$
\int_{\R^3} U_{x,\lambda}^4 \partial_\lambda U_{x,\lambda}\,dy = 2\pi\,\lambda^{-3/2} \int_0^\infty \frac{(1-t^2)t^2\,dt}{(1+t^2)^{7/2}} = -\frac{2\pi}{15} \, \lambda^{-3/2} \,.
$$

\emph{Proof of \eqref{num4}.   } 
Since $PU_{x,\lambda}=U_{x,\lambda}- \lambda^{-1/2} H_0(x,\cdot)-f_{x,\lambda}$,
\[ \lambda^{-1} \int_\Omega a \psi_{x,\lambda} PU_{x,\lambda}\,dy = \lambda^{-1} \int_\Omega a (U_{x,\lambda} - \lambda^{-1/2} H_a(x, y))(U_{x,\lambda} - \lambda^{-1/2} H_0(x, y))\,dy + o(\lambda^{-2}) \,. \]
We have 
\[ \lambda^{-1} \int_\Omega a U_{x,\lambda}^2\,dy = \lambda^{-2} \int_\Omega a(y) \frac{1}{\lambda^{-2} + |x-y|^2} \,dy  = \lambda^{-2} \int_\Omega a(y) \frac{1}{|x-y|^2} \,dy + o(\lambda^{-2}) \]
and, similarly, 
\[ - \lambda^{-3/2} \int_\Omega a U_{x,\lambda} (H_a(x, y) + H_0(x, y))\,dy = - \lambda^{-2} \int_\Omega a(y) \frac{H_a(x, y) + H_0(x, y) }{|x-y|} \,dy + o(\lambda^{-2}). \]
Putting everything together and recalling that $G_a(x,y) = \frac{1}{|x-y|} - H_a(x,y)$, we obtain 
\[  \lambda^{-1} \int_\Omega a \psi_{x,\lambda} PU_{x,\lambda}\,dy = \lambda^{-2} \int_\Omega a(y) G_a(x,y) G_0(x,y) \,dy +o(\lambda^{-2}) =   4 \pi (\phi_a(x) - \phi_0(x))\,\lambda^{-2} + o(\lambda^{-2}) , \]
where the last equality follows from the resolvent identity \eqref{resolvent}.

\emph{Proof of \eqref{num6}.   } 
Since $\|\partial_\lambda f_{x,\lambda}\|_\infty = \mathcal O(\lambda^{-7/2})$ by \cite[Prop.\! 1 (c)]{rey2}, we get, similarly as before,
\[  \int_\Omega a \psi_{x,\lambda} \partial_\lambda PU_{x,\lambda}\,dy =  \int_\Omega a (U_{x,\lambda} - \lambda^{-1/2} H_a(x,y)) (\partial_\lambda U_{x, \lambda} + \frac{1}{2} \lambda^{-3/2} H_0(x, y))\,dy  + o (\lambda^{-2}) \,. \]
We have 
\[  \int_\Omega a U_{x,\lambda} \partial_\lambda U_{x,\lambda}\,dy = \frac12 \lambda^{-2}  \int_\Omega a(y) \frac{\lambda^{-2} - |x-y|^2}{(\lambda^{-2} + |x-y|^2)^{2}} \,dy = -  \frac12 \lambda^{-2} \int_\Omega a(y) \frac{1}{|x-y|^2} \,dy + o(\lambda^{-2}) \]
and, similarly,
\[ -  \lambda^{-1/2} \int_\Omega a H_a(x, y) \partial_\lambda U_{x,\lambda} \,dy =  \frac12 \lambda^{-2} \int_\Omega a(y) \frac{H_a(x,y)}{|x-y|} \,dy  \]
and 
\[   \lambda^{-3/2} \int_\Omega a U_{x,\lambda} H_0(x, y)\,dy =  \lambda^{-2} \int_\Omega a(y) \frac{H_0(x,y)}{|x-y|} \,dy + o(\lambda^{-2}) \,. \]
Putting everything together and using the resolvent identity \eqref{resolvent} as in the proof of \eqref{num4}, we obtain \eqref{num6}. 

This completes the proof of \eqref{eq:n1}.

%%%%%%%%%%%%%%%%

\subsection{Proof of \eqref{eq:d1}}

We have
\begin{align*}
D_1 & = 6\, \beta \lambda^{-1} \int_\Omega \psi_{x,\lambda}^5 PU_{x,\lambda}\,dy + 6\, \gamma \int_\Omega \psi_{x,\lambda}^5 \partial_\lambda PU_{x,\lambda}\,dy \\
& \quad + 15\, \beta^2 \lambda^{-2} \!\int_\Omega \psi_{x,\lambda}^4 PU_{x,\lambda}^2\,dy + 15\,\gamma^2 \!\int_\Omega \psi_{x,\lambda}^4 (\partial_\lambda PU_{x,\lambda})^2\,dy + 30\,\beta\gamma \lambda^{-1} \!\int_\Omega \psi_{x,\lambda}^4 PU_{x,\lambda} \partial_\lambda PU_{x,\lambda}\,dy.
\end{align*}
Therefore \eqref{eq:d1} will follow from the following relations, together with the facts that $\phi_a(x)=o(1)$ by Proposition \ref{thm-q} and that $\beta,\gamma=\mathcal O(1)$ by Lemma \ref{lem-bg},
\begin{align}
\lambda^{-1} \int_\Omega \psi_{x,\lambda}^5 PU_{x,\lambda}\,dy & = \frac{\pi^2}{4}\,\lambda^{-1}  - \frac{4\pi}{3}\, ( 5\,\phi_a(x)+\phi_0(x))\, \lambda^{-2} + o(\lambda^{-2}) \,, \label{eq:den1} \\
\int_\Omega \psi_{x,\lambda}^5 \partial_\lambda PU_{x,\lambda}\,dy
& =  \frac{2\pi }{3}\,(\phi_a(x)+\phi_0(x))\, \lambda^{-2} + o(\lambda^{-2}) \,, \label{eq:den2} \\
\lambda^{-2} \int_\Omega \psi_{x,\lambda}^4 PU_{x,\lambda}^2\,dy & = \frac{\pi^2}{4}\,\lambda^{-2} + o(\lambda^{-2}) \,, \label{eq:den3} \\
\int_\Omega \psi_{x,\lambda}^4 (\partial_\lambda PU_{x,\lambda})^2\,dy & = \frac{\pi^2}{64}\,\lambda^{-2} + o(\lambda^{-2}) \,, \label{eq:den4} \\
\lambda^{-1} \int_\Omega \psi_{x,\lambda}^4 PU_{x,\lambda} \partial_\lambda PU_{x,\lambda}\,dy & = o(\lambda^{-2}) \,. \label{eq:den5}
\end{align}

\emph{Proof of \eqref{eq:den1}.}
We insert $\psi_{x,\lambda}=U_{x,\lambda}-\lambda^{-1/2} H_a(x,\cdot)-f_{x,\lambda}$ and $PU_{x,\lambda}=U_{x,\lambda}- \lambda^{-1/2} H_0(x,\cdot) - f_{x,\lambda}$ to obtain
\begin{align*}
\lambda^{-1} \int_\Omega \psi_{x,\lambda}^5 PU_{x,\lambda}\,dy
& = \lambda^{-1} \int_\Omega U_{x,\lambda}^6\,dy - \lambda^{-3/2} \int_\Omega U_{x,\lambda}^5 (5\, H_a(x,y)+ H_0(x,y)) \,dy + o(\lambda^{-2}) \,.
\end{align*}
For the first term we use \eqref{eq:u6} and for the second term we use Lemma \ref{lem-uh}.

\emph{Proof of \eqref{eq:den2}.}
Similarly as before, we obtain
\begin{align*}
\int_\Omega \psi_{x,\lambda}^5 \partial_\lambda PU_{x,\lambda}\,dy
& = \int_\Omega U_{x,\lambda}^5 \partial_\lambda U_{x,\lambda}\,dy - 5 \lambda^{-1/2} \int_\Omega U_{x,\lambda}^4\partial_\lambda U_{x,\lambda} H_a(x,y) \,dy \\
& \quad + \frac12\,\lambda^{-3/2} \int_\Omega U_{x,\lambda}^5 H_0(x,y)\,dy + o(\lambda^{-2}) \,.
\end{align*}
For the first and the second term we argue as in the proof of \eqref{num5} and for the third one we use Lemma \ref{lem-uh}.

The bounds \eqref{eq:den3}, \eqref{eq:den4} and \eqref{eq:den5} follow from the corresponding relations where $\psi_{x,\lambda}$ and $PU_{x,\lambda}$ are replaced by $U_{x,\lambda}$ and where $\partial_\lambda PU_{x,\lambda}$ is replaced by $\partial_\lambda U_{x,\lambda}$.

This completes the proof of \eqref{eq:d1}.

%%%%%%%%%%%%%%%%%%%
 
\section{Proof of Proposition \ref{prop-app-min}}
\label{sec-app-b} 

In this appendix we provide a proof of the approximate form of almost minimizers. This result is probably well-known to specialists.

\begin{proof}[\bf Proof of Proposition \ref{prop-app-min}] 

\emph{Step 1.} We show that $u_\epsilon\rightharpoonup 0$ in $H^1_0(\Omega)$.

The assumptions imply that $(u_\epsilon)$ is bounded in $H^1_0(\Omega)$ and therefore it has a weak limit point. Let $u_0\in H^1_0(\Omega)$ be such a limit point and write $r_\epsilon := u_\epsilon - u_0$. In the remainder of this step we restrict ourselves to values of $\epsilon$ along which $r_\epsilon\rightharpoonup 0$ in $H^1_0(\Omega)$. By Rellich's compactness theorem $r_\epsilon\to 0$ in $L^2(\Omega)$ and, passing to a subsequence if necessary, we may assume that $r_\epsilon\to 0$ almost everywhere in $\Omega$. By weak convergence in $H^1_0(\Omega)$ and strong convergence in $L^2(\Omega)$ we have
\begin{align*}
3^{-1/2} S^{3/2} +o(1) & = \int_\Omega \left( |\nabla u_\epsilon|^2 + a u_\epsilon^2 + \epsilon V u_\epsilon^2 \right)dx \\
& = \int_\Omega \left( |\nabla u_0|^2 + a u_0^2 \right)dx + \int_\Omega |\nabla r_\epsilon|^2\,dx + o(1) \,.
\end{align*}
Thus,
$$
T := \lim_{\epsilon\to 0} \int_\Omega |\nabla r_\epsilon|^2\,dx
\qquad\text{exists and satisfies}
\qquad
3^{-1/2} S^{3/2} = \int_\Omega \left( |\nabla u_0|^2 + a u_0^2 \right)dx + T \,.
$$
On the other hand, by the almost everywhere convergence and the Br\'ezis--Lieb lemma \cite{BrLi},
$$
(S/3)^{3/2} = \int_\Omega u_\epsilon^6\,dx = \int_\Omega u_0^6\,dx + \int_\Omega r_\epsilon^6\,dx + o(1) \,. 
$$
Thus,
$$
M := \lim_{\epsilon\to 0} \int_\Omega r_\epsilon^6\,dx
\qquad\text{exists and satisfies}
\qquad
(S/3)^{3/2} = \int_\Omega u_0^6\,dx + M \,.
$$
We conclude that
$$
S = \lim_{\epsilon\to 0} \mathcal S_\epsilon[u_\epsilon] = \frac{\int_\Omega \left( |\nabla u_0|^2 + a u_0^2 \right)dx + T}{\left(\int_\Omega u_0^6\,dx + M \right)^{1/3}} \, . 
$$
In the denominator, we bound
\begin{equation}
\label{eq:elementary}
\left(\int_\Omega u_0^6\,dx + M \right)^{1/3} \leq \left(\int_\Omega u_0^6\,dx\right)^{1/3} + M^{1/3}
\end{equation}
and in the numerator we bound $T\geq S M^{1/3}$. Rearranging terms, we thus obtain
$$
S \left(\int_\Omega u_0^6\,dx\right)^{1/3} \geq \int_\Omega \left( |\nabla u_0|^2 + a u_0^2 \right)dx \,.
$$
Since the opposite inequality holds as well by definition of $S(a)$ and the assumption that $S(a)=S$, we need to have, in particular, equality in \eqref{eq:elementary}. It is elementary to see that this holds if and only if either $\int u_0^6\,dx =0$ (that is, $u_0\equiv 0$) or if $M=0$. 

Let us rule out the case $M=0$. If we had $M=0$, then, in particular, $u_0\not\equiv 0$ and therefore $u_0$ would be a minimizer for the $S(a)$ problem. However, as shown by Druet (Step 1 in \cite{dr}), the $S(a)$ problem does not have a minimizer. (Note that this part of Druet's paper does not need any regularity of $a$.) Thus, $M>0$, which, as explained before, implies $u_0\equiv 0$.

\emph{Step 2.} We show that along a subsequence,
\begin{equation}
\label{eq:decomprelim}
u_\epsilon = s\, U_{z_\epsilon,\mu_\epsilon} + \sigma_\epsilon
\end{equation}
with $s\in\{\pm 1\}$, $z_\epsilon\to x_0\in\overline\Omega$, $\mu_\epsilon\, \mathrm{dist}(z_\epsilon,\partial\Omega)\to\infty$ and $\sigma_\epsilon\to 0$ in $\dot H^1(\R^3)$.

Indeed, by Step 1 and Rellich's compactness theorem we have $u_\epsilon \to 0$ in $L^2(\Omega)$ and therefore
$$
\frac{\int_\Omega |\nabla u_\epsilon|^2\,dx}{\left( \int_\Omega u_\epsilon^6\,dx\right)^{1/3}} \to S \,.
$$
Thus, the $u_\epsilon$, extended by zero to functions in $\dot H^1(\R^3)$, form a minimizing sequence for the Sobolev quotient. By a theorem of Lions \cite{Lions} there exist $(z_\epsilon)\subset\R^3$ and $(\mu_\epsilon)\subset\R_+$ such that, along a subsequence, $\mu_\epsilon^{-1/2} u_\epsilon(\mu_\epsilon^{-1} \cdot + z_\epsilon)$ converges in $\dot H^1(\R^3)$ to a function, which is an optimizer for the Sobolev inequality. By the classification of these optimizers (which appears, for instance, in \cite[Cor. I.1]{Lions}) and taking the normalization of the $u_\epsilon$ into account, we can assume, after modifying the $\mu_\epsilon$ and $z_\epsilon$, that
$$
\mu_\epsilon^{-1/2} u_\epsilon(\mu_\epsilon^{-1} \cdot + z_\epsilon) \to s\, U_{0,1}
\qquad\text{in}\ \dot H^1(\R^3)
$$
for some $s\in\{\pm 1\}$. By a change of variables (which preserves the $\dot H^1(\R^3)$ norm) this is the same as \eqref{eq:decomprelim}.

Note that
$$
\int_{\R^3} U^6\,dx = \int_\Omega u_\epsilon^6\,dx = \int_\Omega (s U_{z_\epsilon,\mu_\epsilon}+\sigma_\epsilon)^6\,dx = \int_\Omega U_{z_\epsilon,\mu_\epsilon}^6\,dx + o(1) \,.
$$
Thus, $\mu_\epsilon\to\infty$ and $\mathrm{dist}(z_\epsilon,\Omega)\to 0$. Using, in addition, that the boundary of $\Omega$ is $C^1$, we conclude that $\mu_\epsilon\, \mathrm{dist}(z_\epsilon,\R^3\setminus\Omega)\to\infty$. In particular, after passing to a subsequence, $z_\epsilon\to x_0\in\overline\Omega$.

\emph{Step 3.} We now conclude the proof of the proposition.

Since the remaining arguments are similar to those in \cite[Prop. 2]{rey2} we omit most of the details. As in that paper, the conclusions from Step 2 allow us to apply the result of Bahri--Coron \cite[Prop. 7]{BaCo} and lead to a decomposition
$$
u_\epsilon = \alpha_\epsilon PU_{x_\epsilon,\lambda_\epsilon} + w_\eps
$$
with $x_\epsilon\in\Omega$, bounded $\alpha_\epsilon$ and $w_\epsilon\in T_{x_\epsilon,\lambda_\epsilon}^\bot$ such that $w_\epsilon\to 0$ in $H^1_0(\Omega)$. This implies
$$
\int_\Omega |\nabla (\alpha_\epsilon PU_{x_\epsilon,\lambda_\epsilon})|^2\,dy = \int_\Omega |\nabla u_\epsilon|^2\,dy + o(1) = 3^{-1/2} S^{3/2} + o(1) \,.
$$
By the same argument as in \cite[Prop.~2]{rey2} with $3^{-1/2} S^{3/2}$ instead of $\mu$ on the right side of \cite[(2.18)]{rey2} we infer that $\lambda_\epsilon/\mu_\epsilon + \mu_\epsilon/\lambda_\epsilon + \lambda_\epsilon\mu_\epsilon|x_\epsilon-z_\epsilon|\lesssim 1$. From this we conclude that $\lambda_\epsilon\to\infty$, $x_\epsilon\to x_0$ and $\lambda_\epsilon\mathrm{dist}(x_\eps,\partial\Omega) \to\infty$. Finally, using \cite[(B.2)]{rey2}, $\alpha_\epsilon\to s$. The last relation allows us to replace $w_\epsilon$ by $\alpha_\epsilon w_\epsilon$, which still has the same properties, and obtain the decomposition stated in the proposition. This completes the proof.
\end{proof}

%%%%%%%%%%%%%%%%%%%%%%%%%%%


\begin{thebibliography}{31}

%\vspace{0.15cm}

\bibitem{AmGa} M. Amar, A. Garroni, \textit{$\Gamma$-convergence of concentration problems}. Ann. Sc. Norm. Super. Pisa Cl. Sci. (5) \textbf{2} (2003), no. 1, 151--179. 

\bibitem{AtPe} F. V. Atkinson, L. A. Peletier, \textit{Elliptic equations with nearly critical growth}. J. Differential Equations \textbf{70} (1987), no. 3, 349--365.

\bibitem{Au} Th. Aubin, \textit{Probl\`emes isoperim\'etriques et espaces de Sobolev}. J. Differ. Geometry \textbf{11} (1976), 573--598.

\bibitem{Ba} A. Bahri, \textit{Critical points at infinity in some variational problems}. Pitman Research Notes in Mathematics Series, 182. Longman Scientific \& Technical, Harlow, 1989.

\bibitem{BaCo} A. Bahri, J.-M. Coron, \textit{On a nonlinear elliptic equation involving the critical Sobolev exponent: the effect of the topology of the domain}. Comm. Pure Appl. Math. \textbf{41} (1988), no. 3, 253--294.

\bibitem{Br} H. Br\'ezis, \textit{Elliptic equations with limiting Sobolev exponents—the impact of topology}. Frontiers of the mathematical sciences: 1985 (New York, 1985). Comm. Pure Appl. Math. \textbf{39} (1986), no. S, suppl., S17--S39. 

\bibitem{BrLi} H. Br\'ezis, E. H. Lieb, \textit{A relation between pointwise convergence of functions and convergence of functionals}. Proc. Amer. Math. Soc. \textbf{88} (1983), no. 3, 486--490.

\bibitem{BrNi} H.~Br\'ezis, L.~Nirenberg, \textit{Positive solutions of nonlinear elliptic equations involving critical Sobolev exponents}. Comm. Pure Appl. Math. {\bf 36} (1983), 437--477.

\bibitem{BrPe} H. Br\'ezis, L. A. Peletier, \textit{Asymptotics for elliptic equations involving critical growth}. Partial differential equations and the calculus of variations, Vol. I, 149--192, Progr. Nonlinear Differential Equations Appl., 1, Birkh\"auser Boston, Boston, MA, 1989.

\bibitem{Bu} C. Budd, \textit{Semilinear elliptic equations with near critical growth rates}. Proc. Roy. Soc. Edinburgh Sect. A \textbf{107} (1987), no. 3-4, 249--270.

\bibitem{da} E.B.~Davies, {\em Heat kernels and spectral theory}. Cambridge University Press, Cambridge, 1989.

\bibitem{dr} O.~Druet, \textit{Elliptic equations with critical Sobolev exponents in dimension $3$}. Ann.~I.~H.~Poincar\'e-AN {\bf 19} (2002), 125--142.

\bibitem{DrHeRo} O. Druet, E. Hebey, F. Robert, \textit{Blow-up theory for elliptic PDEs in Riemannian geometry}. Mathematical Notes, 45. Princeton University Press, Princeton, NJ, 2004.

\bibitem{EkFrKo} T. Ekholm, R. L. Frank, H. Kova\v r\'{\i}k, \textit{Weak perturbations of the $p$-Laplacian}. Calc. Var. Partial Differential Equations \textbf{53} (2015), no. 3-4, 781--801.

\bibitem{esp} P.~Esposito, \textit{On some conjectures proposed by Haim Brezis}. Nonlinear Analysis {\bf 54} (2004), 751--759.

\bibitem{Fl} M. Flucher, \textit{Variational problems with concentration}. Progress in Nonlinear Differential Equations and their Applications, 36. Birkh\"auser Verlag, Basel, 1999.

\bibitem{FlGaMu} M. Flucher, A. Garroni, S. M\"uller, \textit{Concentration of low energy extremals: identification of concentration points}. Calc. Var. Partial Differential Equations \textbf{14} (2002), no. 4, 483--516.

\bibitem{Ha} Z.-C. Han, \textit{Asymptotic approach to singular solutions for nonlinear elliptic equations involving critical Sobolev exponent}. Ann. Inst. H. Poincaré Anal. Non Lin\'eaire \textbf{8} (1991), no. 2, 159--174.

\bibitem{HeVa} E. Hebey, M. Vaugon, \textit{From best constants to critical functions}. Math. Z. \textbf{237} (2001), no. 4, 737--767.

\bibitem{LiLo} E. H. Lieb, M. Loss, \textit{Analysis}. Second edition. Graduate Studies in Mathematics, 14. American Mathematical Society, Providence, RI, 2001.

\bibitem{Lions} P.-L. Lions, \textit{The concentration-compactness principle in the calculus of variations. The limit case. I}. Rev. Mat. Iberoamericana \textbf{1} (1985), no. 1, 145--201.

\bibitem{rey1} O.~Rey, \textit{Proof of two conjectures of H. Brezis and L.A. Peletier}. Manuscripta Math. {\bf 65} (1989), 19--37.

\bibitem{rey2} O.~Rey, \textit{The role of the Green's function in a non-linear elliptic equation involving the critical Sobolev exponent}. J. Funct. Anal. {\bf 89} (1990), 1--52. 

\bibitem{Rod} E. Rodemich, \textit{The Sobolev inequality with best possible constant}. Analysis Seminar Caltech, Spring 1966.

\bibitem{Ro} G. Rosen, \textit{Minimum value for $c$ in the Sobolev inequality $\|\phi^3\|\leq c\|\nabla\phi\|^3$}. SIAM J. Appl. Math. \textbf{21} (1971), 30--32.

\bibitem{Sc} R. Schoen, \textit{Conformal deformation of a Riemannian metric to constant scalar curvature}. J. Differential Geom. \textbf{20} (1984), no. 2, 479--495.

\bibitem{Si} B. Simon, \textit{The bound state of weakly coupled Schr\"odinger operators in one and two dimensions}. Ann. Phys. \textbf{97} (1976), 279--288.

\bibitem{St} M. Struwe, \textit{A global compactness result for elliptic boundary value problems involving limiting nonlinearities}. Math. Z. \textbf{187} (1984), no. 4, 511--517.

\bibitem{Tak} F. Takahashi, \textit{On the location of blow up points of least energy solutions to the Brezis--Nirenberg equation}. Funkcial. Ekvac. \textbf{47} (2004), no. 1, 145--166.

\bibitem{Ta} G. Talenti, \textit{Best constants in Sobolev inequality}. Ann. Mat. Pura Appl. \textbf{110} (1976), 353--372.

\bibitem{We} J. Wei, \textit{Asymptotic behavior of least energy solutions to a semilinear Dirichlet problem near the critical exponent}. J. Math. Soc. Japan \textbf{50} (1998), no. 1, 139--153. 

\end{thebibliography}
\end{document}